\documentclass{icmart}
\usepackage{odonnell-icm-style}

\newtheorem{theorem}{Theorem}[section]
\newtheorem{corollary}[theorem]{Corollary}

\newtheorem{proposition}[theorem]{Proposition}

\newtheorem{fact}[theorem]{Fact}
\newtheorem{question}[theorem]{Question}

\theoremstyle{definition}
\newtheorem{definition}[theorem]{Definition}
\newtheorem{remark}[theorem]{Remark}
\newtheorem{example}[theorem]{Example}

\newcommand{\thetagw}{\theta_{\mathrm{GW}}}

\contact[odonnell@cs.cmu.edu]{Computer Science Department, Carnegie Mellon University, Pittsburgh, USA}

\title[Social choice, computational complexity, Gaussian geometry, and Boolean functions]{Social choice, computational complexity, Gaussian geometry, and Boolean functions}

\author[Ryan O'Donnell]
{Ryan O'Donnell\thanks{Supported by NSF grants CCF-1319743 and CCF-1116594.}}

\begin{document}

\begin{abstract}
   We describe a web of connections between the following topics: the mathematical theory of voting and social choice; the computational complexity of the Maximum Cut problem; the Gaussian Isoperimetric Inequality and Borell's generalization thereof; the Hypercontractive Inequality of Bonami; and, the analysis of Boolean functions.  A major theme is the technique of reducing inequalities about Gaussian functions to inequalities about Boolean functions $f \btb$, and then  using induction on~$n$ to further reduce to  inequalities about functions $f \co \bits \to \bits$.  We especially highlight De, Mossel, and Neeman's recent use of this technique to prove the Majority Is Stablest Theorem and Borell's Isoperimetric Inequality simultaneously.
\end{abstract}



\maketitle

\bigskip

(This survey gives only a sketch of various results, and is slightly imprecise in places.  For more details on the topics described herein, see~\cite{OD14}.)

\section{Social choice and Boolean functions}              \label{sec:arrow}

We begin by discussing a problem concerning voting. This will motivate for us certain definitions involving \emph{Boolean functions}; i.e.,~functions $f \btb$ (or  more generally, $f\btR$) whose domain consists of \emph{$n$-bit strings}.  Suppose we have an election with~$n$ voters and $2$~candidates, named~$-1$ and~$1$.  A \emph{voting rule} is simply any Boolean function $f \btb$, mapping the voters' votes to the winner of the election.  The \emph{majority rule} $\Maj_n \btb$, defined (for $n$~odd) by $\Maj_n(x) = \sgn(x_1 + x_2 + \cdots + x_n)$, is perhaps the most natural and mathematically elegant voting rule, but a variety of others are used in practice. Several countries (the US and the UK, for example) elect their head of state via a two-level (weighted-)majority scheme.  Other countries, unfortunately, have been known to use a \emph{dictator} rule: $f(x) = x_i$ for some dictator $i \in [n]$.  The mathematical field of \emph{social choice} is concerned with the properties of various voting rules; for a survey, see e.g.~\cite{BGR09}.

Let's now imagine a twist on the scenario: The $n$ voters decide on their votes, $x = (x_1, \dots, x_n) \in \bn$.
However, due to faulty voting machines, each vote is independently \emph{misrecorded} with probability~$\delta \in [0,1]$.  We denote the resulting list of votes by $\by \in \bn$, and call it a \emph{noisy copy} of the original votes~$x$.  We now ask: \emph{What is the probability that the noise affects the outcome of the election?  How does this probability depend on the voting rule~$f$?}  To answer this question we also need a probabilistic model for how the original votes are cast. We make the simplest possible assumption --- that they are uniformly random, denoted $\bx \sim \bn$.  In the social choice literature this is called the Impartial Culture Assumption~\cite{GK68}.  Let's introduce some mathematical notation for our scenario, using the more convenient parameter $\rho = 1-2\delta \in [-1,1]$:
\begin{definition}
    Given $x \in \bn$ and $\rho \in [-1,1]$, we say that the random vector $\by$ is a \emph{$\rho$-correlated copy} of~$x$ if each coordinate $\by_i$ is independently set to~$x_i$ with probability $\half(1+\rho)$ and set to $-x_i$ with probability $\half(1-\rho)$.  (For the more common case of $\rho \geq 0$, this is equivalent to setting $\by_i = x_i$ with probability~$\rho$ and making~$\by_i$ uniformly random with probability $1-\rho$.)  When $\bx \sim \bn$ is uniformly random and $\by$ is a $\rho$-correlated copy of~$\bx$, we call $(\bx, \by)$ a \emph{$\rho$-correlated random pair of strings}.  Note that this is actually symmetric in $\bx$ and $\by$; an alternative definition is that each pair $(\bx_i, \by_i) \in \bits^2$ is chosen independently with $\E[\bx_i] = \E[\by_i] = 0$ and $\E[\bx_i\by_i] = \rho$.
\end{definition}
\begin{definition}
    For $\rho \in [-1,1]$, the operator $\T_\rho$ acts on Boolean functions ${f \btR}$ via
    \[
        \T_\rho f(x) = \E_{\by \text{ a $\rho$-correlated copy of~$x$}}[f(\by)].
    \]
    We also define the \emph{noise stability of~$f$ at~$\rho$} to be
    \[
        \Stab_\rho[f] = \E_{\bx \sim \bn}[f(\bx) \cdot \T_\rho f(\bx)] = \E_{\substack{(\bx, \by) \textnormal{ $\rho$-correlated} \\ \textnormal{strings}}}[f(\bx)f(\by)].
    \]
    Note that in the special case $f \btb$,
    \[
        \Stab_\rho[f] = 1 - 2\Pr_{\substack{(\bx, \by) \textnormal{ $\rho$-correlated} \\ \textnormal{strings}}}[f(\bx) \neq f(\by)].
    \]
\end{definition}
Returning to the election scenario in which the voters' votes are misrecorded with probability~$\delta$, we see that the probability this affects the outcome of the election is precisely $\half - \half \Stab_{1-2\delta}[f]$.  Thus the voting rules that minimize this probability are precisely those which maximize the noise stability $\Stab_{1-2\delta}[f]$.

Let's focus on the more natural case of $0 < \rho < 1$, i.e., $0 < \delta < \half$.  It's obvious that the Boolean functions $f \btb$ that maximize $\Stab_\rho[f]$ are precisely the two constant functions $f(x) = \pm 1$.  These functions are highly unfair as voting rules, so it's natural to make an assumption that rules them out.  One common such assumption is that $f$ is \emph{unbiased}, meaning $\E[f(\bx)] = 0$; in other words, the two outcomes $\pm 1$ are equally likely when the voters vote uniformly at random.  A stronger, but still very natural, assumption is that~$f$ is \emph{odd}, meaning $f(-x) = -f(x)$. In the social literature this is called \emph{neutrality}, meaning that the voting rule is not affected by changing the names of the candidates.

We might now ask which \emph{unbiased} functions $f \btb$ maximize $\Stab_\rho[f]$.  This problem can be solved easily using \emph{Fourier analysis of Boolean functions}, the basic facts of which we now recall:
\begin{fact}
    Any $f \btR$ can be uniquely expressed as a multilinear polynomial,
    \[
        f(x) = \sum_{S \subseteq [n]} \wh{f}(S) \prod_{i \in S} x_i.
    \]
    This is called the \emph{Fourier expansion} of $f$, and the coefficients $\wh{f}(S) \in \R$ are called the \emph{Fourier coefficients} of~$f$.  We have \emph{Parseval's formula},
    \[
        \E_{\bx \sim \bn} [f(\bx) g(\bx)] = \sum_{S \subseteq [n]} \wh{f}(S)\wh{g}(S).
    \]
    In particular, if $f \btb$ then $\sum_S \wh{f}(S)^2 = 1$.
\end{fact}
\begin{fact}                                    \label{fact:Trho-fourier}
    The Fourier expansion of $\T_\rho f$ is
    \[
        \T_\rho f(x) = \sum_{S \subseteq [n]} \rho^{|S|} \wh{f}(S) \prod_{i \in S} x_i
    \]
    and hence $\Stab_\rho[f] = \sum_{S} \rho^{|S|} \wh{f}(S)^2$.
\end{fact}
Using these facts, the following is an exercise:
\begin{fact}                                        \label{fact:stab-max}
    Assume $0 < \rho < 1$. Then $\Stab_\rho[f] \leq \rho$ holds for all unbiased ${f \btb}$, with equality iff $f$ is a (possibly negated) dictator function, $f(x) = \pm x_i$. Furthermore, $\Stab_{-\rho}[f] \geq -\rho$ holds for \emph{all} ${f \btb}$, not necessarily unbiased, with the same equality conditions.
\end{fact}

This conclusion is somewhat disappointing from the standpoint of election fairness; it says that if our goal is to choose a voting rule that minimizes the effect of misrecorded votes (assuming $0 < \delta < \frac12$), the ``best'' choice is dictatorship (or negated-dictatorship).

Incidentally, this is precisely the disappointment that occurs in \emph{Arrow's Theorem}~\cite{Arr50}, the seminal result in social choice theory.  In brief, Arrow's Theorem is concerned with what happens when~$n$ voters try to rank \emph{three} candidates by means of holding three pairwise elections using Boolean voting rule~$f$.  The well-known \emph{Condorcet Paradox}~\cite{dC85} is that for some~$f$ --- including $f = \Maj_n$ --- it is possible to get an ``irrational'' outcome in which the electorate prefers Candidate~$A$ to Candidate~$B$, prefers Candidate~$B$ to Candidate~$C$, and prefers Candidate~$C$ to Candidate~$A$.  Arrow showed that the only~$f$'s which \emph{always} yield ``rational'' outcomes are dictators and negated-dictators.  Kalai~\cite{Kal02} gave a very elegant Fourier-analytic proof of Arrow's Theorem by noting that when the voters' individual rankings are uniformly random, the probability of a rational outcome is precisely $\frac34 - \frac34 \Stab_{-\frac13}[f]$ (which also equals $\frac34 + \frac34 \Stab_{\frac13}[f]$ for odd~$f$).  Then Arrow's conclusion follows from Fact~\ref{fact:stab-max}.  Kalai also obtained a robust version of Arrow's Theorem by using the \emph{FKN Theorem}~\cite{FKN02} from the analysis of Boolean functions: Any $f$ that achieves a rational outcome with probability at least $1-\delta$ must agree with some (negated-)dictator on all but an $O(\delta)$-fraction of inputs.

Just as we ruled out constant functions~$f$ by insisting on unbiasedness, we might also try to rule out dictatorships (and similar functions) by insisting that~$f$ give only negligible \emph{influence} to each individual voter.  Here we refer to the following definitions:
\begin{definition}
    Let $f \btR$.  For $i \in [n]$, the \emph{(discrete) $i$th derivative} is
    \[
        \D_i f(x) = \tfrac{f(x_1, \dots, x_{i-1}, 1, x_{i+1}, \dots, x_n) - f(x_1, \dots, x_{i-1}, -1, x_{i+1}, \dots, x_n)}{2} = \sum_{S \ni i} \wh{f}(S) \prod_{j \in S \setminus \{i\}} x_j.
    \]
    The \emph{$i$th influence of~$f$} is
    \[
        \Inf_i[f] = \E_{\bx \sim \bn}[\D_if(\bx)^2] = \sum_{S \ni i} \wh{f}(S)^2.
    \]
    Note that when $f \btb$ we also have
    \[
        \Inf_i[f] = \Pr_{\bx \sim \bn}[f(\bx) \neq f(\bx_1, \dots, \bx_{i-1}, -\bx_i, \bx_{i+1}, \dots, \bx_n)].
    \]
\end{definition}
If $f \btb$ is a voting rule, $\Inf_i[f]$ represents the probability that the $i$th voter's vote is pivotal for the outcome.  (This notion was originally introduced by the geneticist Penrose~\cite{Pen46}; it was independently popularized in the social choice literature by the lawyer Banzhaf~\cite{Ban65}.)  The $i$th influence also has an interpretation in terms of the ``geometry'' of the discrete cube graph: if we think of $f \btb$ as the indicator of a vertex set $A \subseteq \bn$, then $\Inf_i[f]$ is fraction of edges in the $i$th coordinate direction that are on $A$'s boundary.

In the interest of fairness, one might want to disallow voting rules $f \btb$ that give unusually large influence to any one voter.  This would disqualify a dictator voting rule like $f(x) = x_i$ since it has $\Inf_i[f] = 1$ (which is maximum possible).  On the other hand, the majority voting rule is quite fair in this regard, since all of its influences are quite small: using Stirling's formula one can compute $\Inf_i[\maj_n] \sim \sqrt{\frac{2}{\pi}} \frac{1}{\sqrt{n}} \xrightarrow{n \to \infty} 0$ for all $i \in [n]$.

We can now ask a question that will occupy us for a significant portion of  this survey:

\begin{question}                                        \label{ques:misc}
    Let ${0 < \rho < 1}$.  Assume $f \btb$ is unbiased and satisfies $\max_i \{\Inf_i[f]\} \leq o_n(1)$. How large can $\Stab_\rho[f]$ be?
\end{question}
We can think of this question as asking for the ``fair'' voting rule that minimizes the effect of misrecorded votes in a noisy election.  Alternatively, the case of $\rho = \frac13$ corresponds to asking for the ``fair'' odd voting rule which maximizes the probability of a ``rational'' outcome in the context of Arrow's Theorem.

Since majority rule seems like a fair voting scheme, it's natural to ask how well it does.  For $n \to \infty$, this can be estimated using the Central Limit Theorem:
\begin{align*}
    \Stab_\rho[\Maj_n] =& \E_{\substack{(\bx, \by) \textnormal{ $\rho$-correlated} \\ \textnormal{strings}}}\left[\sgn\left(\tfrac{\bx_1 + \cdots + \bx_n}{{\sqrt{n}}}\right)\sgn\left(\tfrac{\by_1 + \cdots + \by_n}{{\sqrt{n}}}\right)\right] \\
    \xrightarrow{n \to \infty}& \E_{\substack{(\bz, \bz') \textnormal{ $\rho$-correlated} \\ \textnormal{Gaussians}}}[\sgn(\bz)\sgn(\bz')] = 1 - 2\Pr[\sgn(\bz) \neq \sgn(\bz')],
\end{align*}
where we say $(\bz, \bz')$ is a \emph{$\rho$-correlated pair of Gaussians} if the random variables $\bz,\bz'$ are joint standard normals with $\E[\bz \bz'] = \rho$. An equivalent definition is that ${\bz = \la \vec{u}, \vec{\bg}\ra}$ and $\bz' = \la \vec{v}, \vec{\bg} \ra$, where $\vec{\bg}$ is drawn from the standard $d$-dimensional Gaussian distribution~$\gamma_d$ and $\vec{u}, \vec{v} \in \R^d$ are any two unit vectors satisfying ${\la \vec{u}, \vec{v} \ra = \rho}$.  (In particular, we can take $\bz = \vec{\bg}_1$, $\bz' = \rho \vec{\bg}_1 + \sqrt{1-\rho^2} \vec{\bg}_2$.) Using this latter definition, it's not hard to verify the following old~\cite{She99} fact:
\begin{proposition}[Sheppard's Formula]                                     \label{prop:sheppard}
    If $(\bz, \bz')$ are $\rho$-correlated Gaussians, $-1 \leq \rho \leq 1$, then $\Pr[\sgn(\bz) \neq \sgn(\bz')] = \tfrac{1}{\pi} \arccos \rho$.
\end{proposition}
Taking care with the error term in the Central Limit Theorem, one may deduce:
\begin{proposition}                                     \label{prop:maj-stab}
    For fixed $-1 < \rho < 1$,
    \[
        \Stab_\rho[\Maj_n] = 1 - \tfrac{2}{\pi} \arccos \rho + O(\tfrac{1}{\sqrt{n}}).
    \]
\end{proposition}
\noindent As a corollary, the probability of a ``rational'' outcome when using $\maj_n$ in a three-way election tends to $\frac{3}{2\pi} \arccos(-\frac13) \approx 91\%$, a fact known as \emph{Guilbaud's Theorem}~\cite{Gui52}.\\

Is there a ``fair'' voting rule with even higher noise stability?  In 2004, Khot et~al.~\cite{KKMO04,KKMO07} conjectured the result below, stating that majority essentially gives the best possible answer to Question~\ref{ques:misc}.  A year later their conjecture was proven by Mossel et~al.~\cite{MOO05,MOO10}:
\begin{theorem}[``Majority Is Stablest Theorem'']                                     \label{thm:mist}
    Fix $0 < \rho < 1$.  Assume ${f \btI}$ satisfies $\E[f(\bx)] = 0$ and $\max_i \{\Inf_i[f]\} \leq \eps$.  Then
    \[
        \Stab_\rho[f] \leq 1 - \tfrac{2}{\pi} \arccos \rho + o_\eps(1).
    \]
    (Furthermore, for $-1 < \rho < 0$ the inequality holds in reverse and the hypothesis $\E[f(\bx)] = 0$ is unnecessary.)
\end{theorem}

Peculiarly, the motivation in Khot et~al.~\cite{KKMO04} for conjecturing the above had nothing to do with social choice and voting.  Instead, the conjecture was precisely what was needed to establish the computational complexity of finding approximately maximum \emph{cuts} in graphs.  We discuss this motivation next.

\section{The computational complexity of Max-Cut}              \label{sec:max-cut}

The \emph{Max-Cut} problem is the following fundamental algorithmic task: Given as input is an undirected graph $G = (V,E)$.  The goal is to find a partition ${V = V^+ \cup V^-}$ so as to maximize the fraction of \emph{cut} edges.  Here we say $e \in E$ is ``cut'' if it has one endpoint in each of $V^\pm$.  We write $\Opt(G)$ to denote the value of the best possible solution; i.e., the maximum fraction of edges in~$G$ that can be cut.  For example, $\Opt(G) = 1$ iff $G$ is bipartite.

Unfortunately, the Max-Cut problem is known to be \emph{$\NP$-hard}~\cite{Kar72}.  This means that there is no efficient (i.e., $\poly(|V|)$-time) algorithm for determining $\Opt(G)$, assuming the well-believed $\PTIME \neq \NP$ Conjecture.  Under the closely related \linebreak $\coNP \neq \NP$ Conjecture, we can also state this difficulty as follows: It is \emph{not} true that whenever~$G$ is a graph satisfying $\Opt(G) \leq \beta$, there is a short (i.e., $\poly(|V|)$-length) proof of the statement ``$\Opt(G) \leq \beta$''.

Max-Cut is perhaps the simplest nontrivial \emph{constraint satisfaction problem (CSP)}.  Rather than formally defining this class of problems, we'll simply give two more examples.  In the \emph{Max-3Lin} problem, given is a system of equations over~$\F_2$, each of the form ``$x_{i_1} + x_{i_2} + x_{i_3} = b$''; the task is to find an assignment to the variables $x_1, \dots, x_n$ so as to maximize the fraction of satisfied equations.  In the \emph{Max-3Coloring} problem, given is an undirected graph; the task is to color the vertices using~$3$ colors so as to maximize the fraction of bichromatic edges.  

For all of these CSPs the task of determining $\Opt(\cdot)$ is $\NP$-hard.  One way to cope with this difficulty is to seek \emph{approximation algorithms}:
\begin{definition}
    Let $0 \leq \alpha \leq \beta \leq 1$.  Algorithm $\calA$ is said to be \emph{$(\alpha,\beta)$-approximating} for a certain CSP (e.g., Max-Cut) if it has the following guarantee: For every input~$G$ satisfying $\Opt(G) \geq \beta$, the algorithm finds a solution of value at least~$\alpha$.  If $\calA$ is a randomized algorithm, we allow it to achieve value at least~$\alpha$ \emph{in expectation}.  Note that a fixed~$\calA$ may be $(\alpha,\beta)$-approximating for many pairs $(\alpha, \beta)$ simultaneously.
\end{definition}
\begin{example}
    There is a simple greedy algorithm that is $(1,1)$-approximating for Max-Cut; i.e., given a bipartite~$G$, it finds a bipartition.  Similarly, one can efficiently $(1,1)$-approximate Max-3Lin using Gaussian elimination.  On the other hand, $(1,1)$-approximating Max-3Coloring --- i.e., validly $3$-coloring $3$-colorable graphs --- is $\NP$-hard.  For Max-3Lin the near-trivial algorithm of outputting $x_1 = \cdots = x_n = B$, where $B$ is the more common ``right-hand side'' of the system, is a $(\half, \beta)$-approximation for every~$\half \leq \beta \leq 1$. One can also get an efficient $(\half, \beta)$-approximation for Max-Cut (for any~$\beta$) either by a simple greedy algorithm, or by outputting a \emph{random} partition $V = V^+ \cup V^-$.  
    The classical statement that ``Max-Cut is $\NP$-hard'' is equivalent to stating that \emph{there exists} $\half < \beta < 1$ such that $(\beta, \beta)$-approximating Max-Cut is $\NP$-hard (in fact, this is true for all $\half < \beta < 1$).
\end{example}
In the case of Max-3Lin, it is a rather astonishing fact that the trivial approximation algorithms mentioned above are best possible assuming $\PTIME \neq \NP$; this is a celebrated result of H{\aa}stad~\cite{Has97,Has01} combining ``PCP technology''~\cite{FGL+96,AS98,ALM+98,BGS98} and Fourier analysis of Boolean functions:
\begin{theorem}                                     \label{thm:hastad}
    For any $\delta > 0$, it's $\NP$-hard to $(\half + \delta, 1 - \delta)$-approximate Max-3Lin.
\end{theorem}
For quite a long time, it was not known how to do any better even for the much simpler problem of Max-Cut.  This changed in 1994 with the famous and sophisticated result of Goemans and Williamson~\cite{GW94,GW95}  (see also~\cite{DP93}):
\begin{theorem}                                     \label{thm:gw}
    There is an efficient algorithm that $(\frac{\theta}{\pi}, \half - \half \cos \theta)$-approximates Max-Cut for every $\theta \in [\thetagw, \pi]$, where $\thetagw \approx .74\pi$ is the positive solution of ${\tan(\frac{\theta}{2}) = \theta}$.  E.g., the Goemans--Williamson algorithm simultaneously $(\frac34, \frac12 + \frac{1}{2\sqrt{2}})$-approximates, 
    $(\frac45, \frac58 + \frac{\sqrt{5}}{8})$-approximates, 
    and $(1-\frac2{\pi}\sqrt{\eps} -o(\sqrt{\eps}), 1-\eps)$-approximates Max-Cut.
\end{theorem}
\noindent (Variants of the Goemans--Williamson algorithm that perform well for $\theta < \thetagw$ are also known.)



\vspace{-1.4in}
\begin{center}
    \includegraphics[width=.75\textwidth]{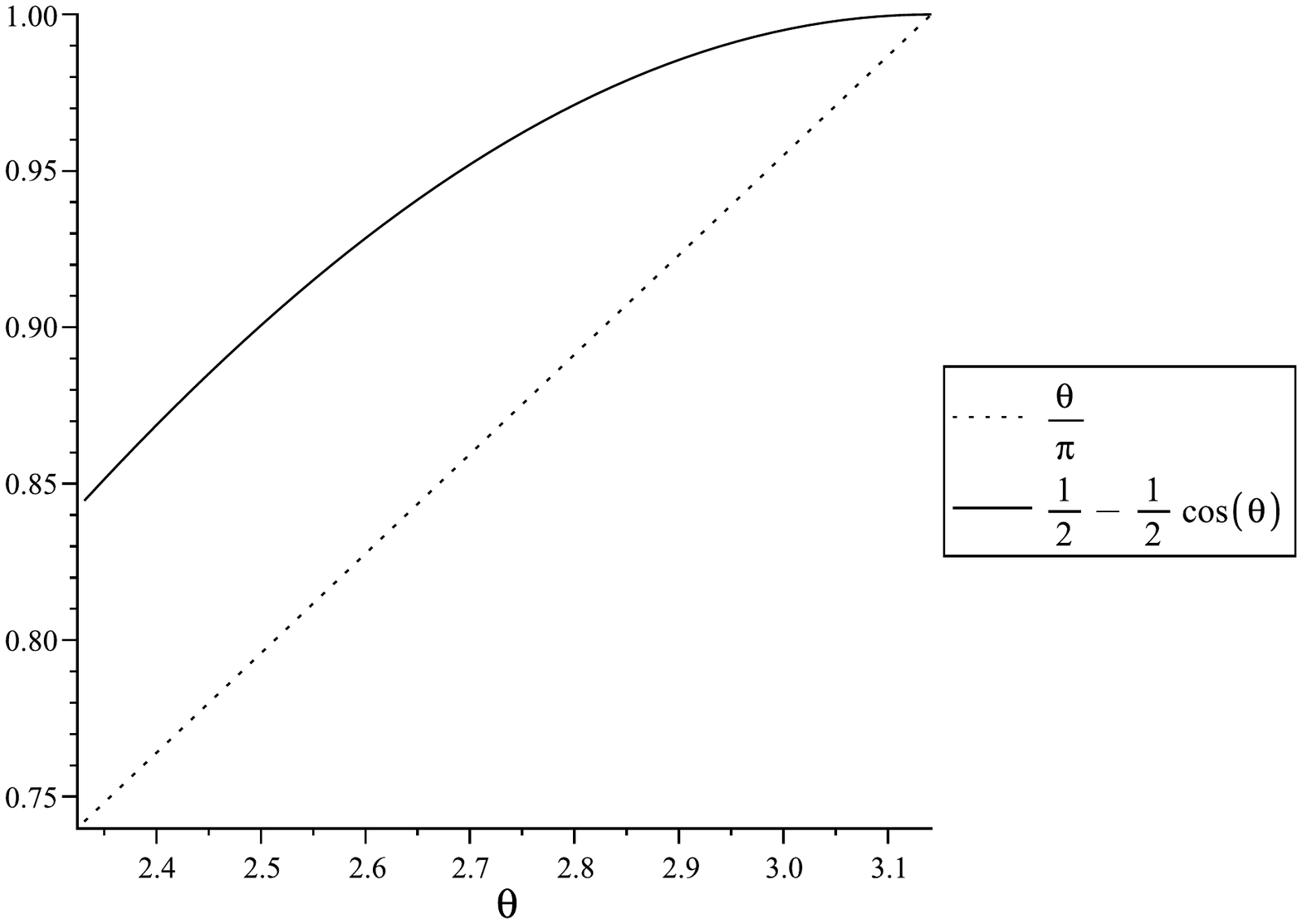}
\end{center}


\vspace{-1.2in}

Briefly, the algorithm works as follows:  Given a graph $G = (V,E)$, one considers the following \emph{semidefinite programming} optimization problem:
\begin{equation} \label{eqn:gw-sdp} \tag{SDP}
\begin{aligned}
    \SDPOpt(G)\ \ =\ \ \text{max}& \quad  \avg_{(v,w) \in E} \left[\half - \half \la \vec{U}(v), \vec{U}(w) \ra\right] \\
    \text{subject to}& \quad \vec{U} \co V \to S^{d-1}.
\end{aligned}
\end{equation}
Here one also maximizes over all $d \in \Z^+$, although one can show that it suffices to take $d = |V|$.  Essentially, the optimization problem~\eqref{eqn:gw-sdp} seeks to assign a unit vector to each vertex in~$V$ so that edges in~$G$ are spread as far apart as possible.  It's easy to see that if $d$ is fixed to~$1$ (so that $\vec{U} \co V \to \{-1,1\}$) then~\eqref{eqn:gw-sdp} is identical to the Max-Cut problem; therefore $\Opt(G) \leq \SDPOpt(G)$ always.  Surprisingly, although computing $\Opt(G)$ is intractable, one can efficiently compute $\SDPOpt(G)$.  (Roughly speaking, the reason is that if we introduce real variables $\rho_{vw} = \la \vec{U}(v), \vec{U}(w) \ra$, then~\eqref{eqn:gw-sdp} is equivalent to maximizing a linear function of the $\rho_{vw}$'s over an explicit convex subset of $\R^{|V| \times |V|}$, namely the set of all positive semidefinite matrices $R = (\rho_{vw})_{v,w \in V}$ with $1$'s on the diagonal.)

Thus~\eqref{eqn:gw-sdp} gives us an efficiently-computable upper bound on~$\Opt(G)$. One may hope that it is a relatively ``good'' upper bound, and that furthermore one can prove this constructively by providing an efficient algorithm which converts the optimum ``vector solution'' $(\vec{U}^*(v))_{v \in V}$ to a good ``$\pm 1$ solution'' $(U^*(v))_{v \in V}$ --- i.e., a good bipartition of~$V$.  Goemans and Williamson fulfilled this hope, as follows: Their algorithm first chooses $\vec{\bg}$ to be a standard $d$-dimensional Gaussian and then it outputs the bipartition of~$G$ defined by $U^*(v) = \la \vec{U}^*(v), \vec{\bg} \ra$.  Using Sheppard's Formula, it's not hard to show that this establishes Theorem~\ref{thm:gw}.

The Goemans--Williamson algorithm was originally considered to be quite complex for such a simple CSP as Max-Cut; furthermore, its approximation guarantee seemed quite peculiar.  More than one paper~\cite{Fei99,FS02} suggested the research goal of improving this approximation guarantee.  Furthermore, the best known $\NP$-hardness result for the problem (from~\cite{Has97,TSSW00}) does not match the algorithm.  For example, it's known that $(.875 + \delta, .9)$-approximating Max-Cut is $\NP$-hard for all $\delta > 0$, and the Goemans--Williamson algorithm achieves $(\alpha, .9)$-approximation for $\alpha = 1-\frac{1}{\pi} \arccos\frac{4}{5} \approx .795$.  But whether cutting $80\%$ of the edges in a graph $G$ with $\Opt(G) = 90\%$ is polynomial-time solvable or is $\NP$-hard is unknown.

Nevertheless, in 2004 Khot et al.~\cite{KKMO04} obtained the following ``surprising''~\cite{Joh06} result: Under the \emph{Unique Games Conjecture}~\cite{Kho02} (a notorious conjecture in computational complexity not related to Max-Cut), the Majority Is Stablest Theorem implies that there is no efficient algorithm beating the Goemans--Williamson approximation guarantee (at least for $\theta \in [\thetagw, \pi]$; see~\cite{OW08} for optimal results when $\theta < \thetagw$).  We remark that the while the Unique Games Conjecture is believable, its status is vastly more uncertain than the $\PTIME \neq \NP$ conjecture.

Let us briefly explain what the Majority Is Stablest Theorem has to do with the complexity of the Max-Cut problem.  As shown in~\cite{KKMO04}, the advantage of the Unique Games Conjecture (as opposed to just the $\PTIME \neq \NP$ assumption) is that it makes the ``H{\aa}stad PCP technology'' much easier to use.  Very roughly speaking, it implies that to establish intractability of beating $(\frac{\theta}{\pi}, \half - \half \cos \theta)$-approximation, it suffices to find certain so-called ``gadget graphs'' for the Max-Cut problem.  Precisely speaking, these gadget graphs need to have the following properties:
\begin{itemize}
    \item The vertex set $V$ should be $\bits^n$.  (As a consequence, bipartitions of $V$ correspond to Boolean functions $f \btb$.)
    \item The bipartitions given by the $n$ ``dictators'' $f(x) = x_i$ should each cut at least a $\half - \half \cos \theta$ fraction of the edges.
    \item Any bipartition which is not ``noticeably correlated'' with a dictator partition should not cut ``essentially more'' than a $\frac{\theta}{\pi}$ fraction of the edges.  More precisely, if $f \btb$ is any bipartition of~$V$ with $\max_i \{\Inf_i[f]\} \leq \eps$, then the fraction of edges it cuts is at most $\frac{\theta}{\pi} + o_\eps(1)$.
\end{itemize}
Actually, it's also acceptable for these gadgets to be \emph{edge-weighted} graphs, with nonnegative edge-weights summing to~$1$.  Khot et al.~suggested using the \emph{noisy hypercube} graph on vertex set $\{-1,1\}^n$, in which the weight on edge $(u,v) \in \bn \times \bn$ is precisely $\Pr[\bx = u, \by = v]$ when $(\bx,\by)$ are a $(\cos \theta)$-correlated random strings (note that $\rho = \cos \theta < 0$ for $\theta \in [\thetagw, \pi]$).  Such gadget graphs have the first two properties above, and the Majority Is Stablest Theorem precisely implies that they also have the third property.  It's somewhat surprising that the technical properties required for this Unique Games/PCP-based hardness result correspond so perfectly to a natural problem about voting theory.

Thus subject to the Unique Games Conjecture, no efficient algorithm can improve on the Goemans--Williamson Max-Cut approximation guarantee.  In particular, this means that there must be infinite families of graphs on which the Goemans--Williamson algorithm performs no better than the guarantee established in Theorem~\ref{thm:gw}.  As first shown by Karloff~\cite{Kar99}, the noisy hypercube graphs~$G$ also serve as examples here: Though they have $\Opt(G) = \half - \half \cos \theta$, one optimal solution of~\eqref{eqn:gw-sdp} for these graphs is  $\vec{U}^*(v) = v/\sqrt{d}$, and applying the Goemans--Williamson algorithm to these vectors will indeed give a bipartition cutting only a $\frac{\theta}{\pi}$ fraction of edges in expectation.

Before turning our attention more fully to the Majority Is Stablest Theorem, we should mention a far-reaching generalization of the above-described work in complexity theory, namely the Raghavendra Theory of CSP approximation.  Raghavendra~\cite{Rag08} showed that for \emph{all} CSPs (not just Max-Cut), the natural analogue of the Goemans--Williamson SDP algorithm has optimal approximation guarantee among all efficient algorithms, subject to the Unique Games Conjecture.  This theory will be discussed further in our concluding Section~\ref{sec:conclusion}.

\section{Borell's Isoperimetric Inequality}     \label{sec:borell}

The Majority Is Stablest Theorem concerns Boolean functions, but thanks to the Central Limit Theorem it includes as a ``special case'' a certain inequality concerning \emph{Gaussian geometry} first proved by Borell~\cite{Bor85}.  (In this field, the idea that Boolean inequalities imply Gaussian inequalities dates back to the work of Gross~\cite{Gro75} on the Log-Sobolev Inequality.) To state this Gaussian inequality we first make some definitions:
\begin{definition}
    Let $\vphi$ and $\Phi$ denote the standard Gaussian pdf and cdf, respectively.   Given $z \in \R^d$ and $\rho \in [-1,1]$, we say that the random vector~$\bz'$ is a \emph{$\rho$-correlated Gaussian copy} of~$z$ if $\bz'$ has the distribution $\rho z + \sqrt{1-\rho^2} \bg$, where~$\bg$ is a standard $d$-dimensional Gaussian random vector.  When~$\bz$ is itself a standard $d$-dimensional Gaussian and $\bz'$ is a $\rho$-correlated Gaussian copy, we call $(\bz, \bz')$ a \emph{$\rho$-correlated $d$-dimensional Gaussian pair}.  An equivalent definition is that each pair of random variables $(\bz_i, \bz'_i)$ is a $\rho$-correlated pair of Gaussians (as defined in Section~\ref{sec:arrow}) and the pairs are independent for $i \in [d]$. Note that $(\bz,\bz')$ has the same distribution as $(\bz', \bz)$.
\end{definition}
\begin{remark}                                  \label{rem:sphere-idea}
    The distribution of a $\rho$-correlated $d$-dimensional Gaussian pair $(\bz, \bz')$ is also rotationally symmetric in $\R^d$.  Note that for large~$d$ we'll have $\|\bz\|, \|\bz'\| \sim \sqrt{d}$ and $\la \bz, \bz' \ra \sim \rho d$.  Thus an intuitive picture to keep in mind when~$d$ is large is that $(\bz, \bz')$ is roughly distributed as a uniformly random pair of vectors of length $\sqrt{d}$ and angle $\arccos \rho$.
\end{remark}
\begin{definition}
    The \emph{Ornstein--Uhlenbeck semigroup} of operators is defined as follows: For $\rho \in [-1,1]$, the
    operator $\U_\rho$ acts on 
    functions $f \co \R^d \to \R$ by
    \[
        \U_\rho f(z) = \E_{\bz' \text{ a $\rho$-correlated Gaussian copy of~$z$}}[f(\bz')].
    \]
    We also define the \emph{Gaussian noise stability of~$f$ at~$\rho$} to be
    \[
        \Stab_\rho[f] = \E_{\substack{(\bz, \bz') \textnormal{ $\rho$-correlated} \\ \textnormal{ $d$-dimensional Gaussian pair}}}[f(\bz)f(\bz')].
    \]
\end{definition}
We can now state the ``Gaussian special case'' of Majority Is Stablest:
\begin{theorem}                                     \label{thm:gaussian-mist}
    Fix $0 < \rho < 1$.  Assume $h \co \R^d \to [-1,1]$ satisfies $\E_{\bz \sim \gamma_d}[h(\bz)] = 0$.  Then its Gaussian noise stability satisfies
    \[
        \Stab_\rho[h] \leq 1 - \tfrac{2}{\pi} \arccos \rho.
    \]
    (Furthermore, for $-1 < \rho < 0$ the inequality holds in reverse and the hypothesis $\E[h] = 0$ is unnecessary.)
\end{theorem}
To obtain Theorem~\ref{thm:gaussian-mist} from the Majority Is Stablest Theorem (at least for ``nice enough''~$h$), we use the fact that Gaussian random variables can be ``simulated'' by sums of many independent $\pm 1$ random bits. More precisely, we can apply Majority Is Stablest to $f \co \bits^{dn} \to [-1,1]$ defined by
\[
    f(x_{1,1}, \dots, x_{d,n}) = h\left(\tfrac{x_{1,1} + \cdots + x_{1,n}}{\sqrt{n}}, \dots, \tfrac{x_{d,1} + \cdots + x_{d,n}}{\sqrt{n}}\right)
\]
and then take $n \to \infty$ and use a $d$-dimensional Central Limit Theorem.  (The assumption and error dependence on the influence bound~$\epsilon$ disappears, because we have $\eps \to 0$ as $n \to \infty$.)  Note that in Section~\ref{sec:arrow} we saw exactly this limiting procedure in the case of $h = \sgn : \R^1 \to \{-1,1\}$ when we computed the limiting (Boolean) noise stability of~$\Maj_n$.

Theorem~\ref{thm:gaussian-mist} was first proved by Borell in 1985~\cite{Bor85}.  (In fact, Borell proved significant generalizations of the theorem, as discussed below.)  In 2005, Mossel et al.~\cite{MOO10} used it to prove the Majority Is Stablest Theorem by reducing the Boolean setting to the Gaussian setting.  The key technical tool here was a ``nonlinear'' version of the Central Limit Theorem called the \emph{Invariance Principle} (see also~\cite{Rot75}).  Briefly, the Invariance Principle implies that if $f \btR$ is a low-degree multilinear polynomial with small influences then the distributions of $f(\bx_1, \dots, \bx_n)$ and $f(\bg_1, \dots, \bg_n)$ are ``close'', where $\bx_1, \dots, \bx_n$ are independent $\pm 1$ random variables and $\bg_1, \dots, \bg_n$ are independent Gaussians.  The Invariance Principle has had many applications (e.g., in combinatorics~\cite{DFR08}, learning theory~\cite{Kan12a}, pseudorandomness~\cite{MZ10}, social choice~\cite{Mos10}, sublinear algorithms~\cite{BO10}, and the Raghavendra Theory of CSPs mentioned at the end of Section~\ref{sec:max-cut}) but we won't discuss it further here.  Instead, we'll outline in Section~\ref{sec:DMN} an alternative, ``purely discrete'' proof of the Majority Is Stablest Theorem due to De, Mossel, and Neeman~\cite{DMN13}.

Let's now look more carefully at the geometric content of Theorem~\ref{thm:gaussian-mist}. Suppose $A \subset \R^d$ is a set with \emph{Gaussian volume} $\gamma_d(A) = \half$.  Applying Theorem~\ref{thm:gaussian-mist} with $h = 1 - 2 \cdot 1_A$, and also writing $\theta = \arccos \rho \in (0,\frac{\pi}{2})$, one obtains the following:
\begin{corollary}                                       \label{cor:easy-borell}
    For $0 \leq \theta \leq \frac{\pi}{2}$ and $A \subseteq \R^d$, define the \emph{rotation sensitivity}
    \[
        \RS_A(\theta) = \Pr_{\substack{(\bz, \bz') \textnormal{ $\cos \theta$-correlated} \\ \textnormal{$d$-dimensional Gaussian pair}}}[1_A(\bz) \neq 1_A(\bz')].
    \]
    Then if $\gamma_d(A) = \half$, we have $\RS_A(\theta) \geq \frac{\theta}{\pi}$.
\end{corollary}
By Sheppard's Formula, equality is obtained if $d = 1$ and $A = (-\infty, 0]$. In fact, by rotational symmetry of correlated Gaussians, equality is obtained when~$A$ is any halfspace through the origin in~$\R^d$.  (Geometrically, it's natural to guess that halfspaces minimize $\RS_A(\theta)$ among sets~$A$ of fixed Gaussian volume, using the intuition from Remark~\ref{rem:sphere-idea}.) 
As shown in~\cite{KO12}, this corollary is quite easy to prove for ``many'' values of~$\theta$:
\begin{proof}[Proof of Corollary~\ref{cor:easy-borell} for $\theta = \frac{\pi}{2\ell}$, $\ell \in \Z^+$]  
Let $\bg, \bg'$ be independent {$d$-dimensional} Gaussians and define $\bz^{(j)} = \cos(j\theta) \bg + \sin(j\theta) \bg'$ for $0 \leq j \leq \ell$.  Then it's easy to check that $(\bz^{(i)}, \bz^{(j)})$ is a $\cos((j-i)\theta)$-correlated Gaussian pair. In particular, $\bz^{(0)}$ and $\bz^{(\ell)}$ are independent.  Now using $\gamma_d(A) = \half$ and a union bound we get
\[
    \half = \Pr[1_A(\bz^{(0)}) \neq 1_A(\bz^{(\ell)})] \leq \sum_{j = 1}^\ell \Pr[1_A(\bz^{(j-1)}) \neq 1_A(\bz^{(j)})] = \ell \cdot \RS_{A}(\theta),
\]
which is the desired inequality.
\end{proof}

Returning to Theorem~\ref{thm:gaussian-mist}, it states that if $(\bz, \bz')$ are $\rho$-correlated $d$-dimensional Gaussians ($0 < \rho < 1$) then halfspaces are the volume-$\half$ sets which maximize $\Pr[\bz, \bz' \in A]$.  In fact, halfspaces are also the optimizers at \emph{any} fixed volume.  Furthermore, if we generalize by looking for sets $A, B$ of fixed volume maximizing $\Pr[\bz \in A, \bz' \in B]$, parallel halfspaces are again best.  These isoperimetric facts (and more) were all originally proved by Borell~\cite{Bor85}:
\begin{theorem}[``Borell Isoperimetric Inequality'']                                     \label{thm:borell}
    Fix $0 < \rho < 1$ and ${0 \leq \alpha, \beta \leq 1}$.  Suppose $A, B \subseteq \R^d$ satisfy $\gamma_d(A) = \alpha$, $\gamma_d(B) = \beta$.  Then if $(\bz, \bz')$ is a $\rho$-correlated $d$-dimensional Gaussian pair,
    \[
        \Pr[\bz \in A, \bz' \in B] \leq \Pr[\bz \in H, \bz' \in H']
    \]
    where $H$ and $H'$ are (any) parallel halfspaces satisfying $\gamma_d(H) = \alpha$, $\gamma_d(H') = \beta$.      (If $-1 < \rho < 0$ then the inequality is reversed.) By rotational symmetry we may assume $H = (-\infty, \Phi^{-1}(\alpha)],\ H' = (-\infty, \Phi^{-1}(\beta)] \subseteq \R$ and thus write the above as
    \[
        \Pr[\bz \in A, \bz' \in B] \leq \Lambda_\rho(\alpha,\beta) \coloneqq \Pr_{\substack{(\bw, \bw') \textnormal{ $\rho$-correlated} \\ \textnormal{Gaussians}}}[\bw \leq \Phi^{-1}(\alpha), \bw' \leq \Phi^{-1}(\beta)].
    \]
    In case $\alpha = \beta = \frac12$, Sheppard's Formula implies
    \[
        \Pr[\bz \in A, \bz' \in B] \leq \Lambda_\rho(\tfrac12,\tfrac12) = \tfrac12 - \tfrac{1}{2\pi} \arccos \rho.
    \]
\end{theorem}
Borell's original proof of this theorem used the Gaussian symmetrization method due to Ehrhard~\cite{Ehr83} and was quite technical.  Four other proofs are known.  Beckner~\cite{Bec92} pointed out that the analogous isoperimetric inequality on the sphere is easy to prove by two-point symmetrization~\cite{BT76}, and the Gaussian result can then be deduced via ``Poincar\'{e}'s limit'' (see~\cite{CL90}).  Mossel and Neeman~\cite{MN12} recently gave a slick proof using semigroup methods, and together with De~\cite{DMN13} they gave another proof via Boolean functions. Finally, Eldan~\cite{Eld13} gave the most recent new proof, using stochastic calculus.

We will describe De, Mossel, and Neeman's Boolean proof of Borell's Isoperimetric Inequality in Section~\ref{sec:DMN}.  It has the advantage that it can be used to prove the Majority Is Stablest Theorem ``at the same time'' (using a few technical tricks from the original Invariance Principle-based proof, including \emph{hypercontractivity}).  But first, we'll spend some time discussing further special cases of Borell's Isoperimetric Inequality.

\section{Hypercontractivity}                \label{sec:hypercon}

Borell's Isoperimetric Inequality is very precise, giving the exact maximal value of $\Pr[\bz \in A, \bz' \in B]$ (when $(\bz,\bz')$ are $\rho$-correlated) for any fixed Gaussian volumes $\gamma_d(A) = \alpha$, $\gamma_d(B) = \beta$.  A small downside is that this maximum value, $\Lambda_\rho(\alpha, \beta)$, does not have a nice closed-form expression except when $\alpha = \beta = \half$.  In the interesting regime of $\alpha, \beta \to 0$, however, we can get a closed form for its asymptotics.  Let's do a rough ``heuristic'' estimation.

Suppose $H, H'$ are parallel halfspaces of ``small'' Gaussian volume $\alpha, \beta$, with $\alpha \leq \beta$.  By rotational symmetry we can assume $H = [a, \infty), H' = [b, \infty) \subset \R$ for some ``large'' values $a \geq b > 0$.  Precisely, we have $a = -\Phi^{-1}(\alpha)$, but speaking roughly we'll express this as $\alpha \approx \exp(-\frac{a^2}{2})$, as this is asymptotically correct up to lower-order factors.  Similarly we'll write $\beta \approx \exp(-\frac{b^2}{2})$.  We are interested in estimating $\Pr[\bg \in H, \bg' \in H']$, where $(\bg, \bg')$ are a $\rho$-correlated Gaussian pair.  We'll actually take $\bg' = \rho \bg + \sqrt{1-\rho^2} \bh$, where $\bh$ is a standard Gaussian independent of~$\bg$.  To start the estimation, by definition we have $\Pr[\bg \in H] \approx \exp(-\frac{a^2}{2})$.  Further, conditioned on $\bg \in H$ we will almost surely have that $\bg$ is only ``barely'' larger than~$a$.  Thus we expect $\bg'$ to be conditionally distributed roughly as $\rho a + \sqrt{1-\rho^2} \bh$. In this case, $\bg'$ will be in $H'$ if and only if $\bh \geq (b - \rho a)/\sqrt{1-\rho^2}$. Under the assumption that $b - \rho a \geq 0$, the probability of this is, roughly again, $\exp(-\frac{(b-\rho a)^2}{2 (1-\rho^2)})$.  All in all, these calculations ``suggest'' that
\[
    \Lambda_\rho(\alpha, \beta) = \Pr[\bg \in H, \bg' \in H'] \approx \exp(-\tfrac{a^2}{2})\exp(-\tfrac{(b-\rho a)^2}{2 (1-\rho^2)}) = \exp\left(-\tfrac12 \tfrac{a^2 - 2\rho a b + b^2}{1-\rho^2}\right)
\]
(under the assumption that $\alpha \approx \exp(-\frac{a^2}{2}) \leq \exp(-\frac{b^2}{2}) \approx \beta$ are ``small'', with $b \geq \rho a$).  Since Borell's Isoperimetric Inequality tells us that parallel halfspaces are maximizers, we might optimistically guess the following:
\begin{theorem}[``Gaussian Small-Set Expansion Theorem'']                            \label{thm:gaussian-sse}
    Let $0 < \rho < 1$.  Let $A, B \subseteq \R^d$ have Gaussian volumes $\exp(-\frac{a^2}{2}), \exp(-\frac{b^2}{2})$, respectively, and assume $0 \leq \rho a \leq b \leq a$.  Then
    \[
        \Pr_{\substack{(\bz, \bz') \textnormal{ $\rho$-correlated} \\ \textnormal{$d$-dimensional Gaussian pair}}}[\bz \in A, \bz' \in B] \leq \exp\left(-\tfrac12 \tfrac{a^2 - 2\rho a b + b^2}{1-\rho^2}\right).
    \]
    In particular, if $A \subseteq \R^d$ has $\gamma_d(A) = \alpha$ then
    \begin{equation}                                \label{eqn:g-sse}
        \Stab_\rho[1_A] \leq \alpha^{\frac{2}{1+\rho}} \iff \Pr_{\substack{(\bz, \bz') \textnormal{ $\rho$-correlated} \\ \textnormal{$d$-dimensional Gaussian pair}}}[\bz' \in A \mid \bz \in A] \leq \alpha^{\frac{1-\rho}{1+\rho}}.
    \end{equation}
\end{theorem}
Indeed this theorem is correct, and it can be formally deduced from Borell's Isoperimetric Inequality.  We'll outline a more direct proof shortly, but first let's discuss its content.  The one-set statement~\eqref{eqn:g-sse} says that if~$A$ is any ``small'' subset of Gaussian space (think of $\alpha$ as tending to~$0$) and $\rho$ is bounded away from~$1$ (say $\rho = 1 - \delta$), then  a $\rho$-noisy copy of a random point in~$A$ will almost certainly (i.e., except with probability $\alpha^{\delta/(2+\delta)}$) be outside~$A$. 

One might ask whether a similar statement is true for subsets of the discrete cube~$\bn$.  As we saw with Majority Is Stablest implying Theorem~\ref{thm:gaussian-mist}, isoperimetric inequalities on the discrete cube typically imply the analogous statement in Gaussian space, by the Central Limit Theorem.  On the other hand, the converse does not generally hold; this is because there are subsets of $\bn$ like the dictators $\{x : x_i = 1\}$, or more generally ``subcubes'' ${\{x : x_{i_1} = \cdots = x_{i_k} = 1\}}$, which have no analogue in Gaussian space.  In particular, one has to rule out dictators using the ``small-influences'' condition in order for the Boolean analogue of Borell's theorem, namely the Majority Is Stablest Theorem, to be true.  However it \emph{is} often true that asymptotic isoperimetric inequalities for ``small'' subsets of Gaussian space also hold in the Boolean setting with no influences assumption; this is because \emph{small} subcubes and \emph{small} Hamming balls (the Boolean analogue of Gaussian halfspaces) have similar isoperimetric properties in $\bn$.  In particular, it turns out that Theorem~\ref{thm:gaussian-sse} holds identically in $\bn$:
\begin{theorem}[``Boolean Small-Set Expansion Theorem'']                            \label{thm:boolean-sse}
    Let $0 < \rho < 1$.  Let $A, B \subseteq \bn$ have volumes $\frac{|A|}{2^n} = \exp(-\frac{a^2}{2})$, $\frac{|B|}{2^n} = \exp(-\frac{b^2}{2})$, and assume $0 \leq \rho a \leq b \leq a$.  Then
    \[
        \Pr_{\substack{(\bx, \by) \textnormal{ $\rho$-correlated strings}}}[\bx \in A, \by \in B] \leq \exp\left(-\tfrac12 \tfrac{a^2 - 2\rho a b + b^2}{1-\rho^2}\right).
    \]
    In particular, if $\frac{|A|}{2^n} = \alpha$ then
    \begin{equation}                                \label{eqn:b-sse}
        \Stab_\rho[1_A] \leq \alpha^{\frac{2}{1+\rho}} \iff \Pr_{\substack{(\bx, \by) \textnormal{ $\rho$-correlated strings} \\ }}[\bx \in A \mid \by \in A] \leq \alpha^{\frac{1-\rho}{1+\rho}}.
    \end{equation}
\end{theorem}
This theorem is formally stronger than its Gaussian counterpart Theorem~\ref{thm:gaussian-sse}, by virtue of the Central Limit Theorem.  In fact, there is a related \emph{functional inequality} which is even stronger; this is the crucial \emph{Hypercontractive Inequality} first proved by Bonami~\cite{Bon70}.
\begin{theorem}[``Boolean Hypercontractive Inequality'']                                \label{thm:hypercon}
    Let $f, g \co \bn \to \R$, let $r, s \geq 0$, and assume $0 \leq \rho \leq \sqrt{rs} \leq 1$.  Then
    \[
        \E_{\substack{(\bx, \by) \textnormal{ $\rho$-correlated}}}[f(\bx) g(\by)] \leq \|f\|_{1+r} \|g\|_{1+s}.
    \]
    (Here we are using $L^p$-norm notation, $\|f\|_p = \E_{\bx \sim \bn} [|f(\bx)|^p]^{1/p}$.)
\end{theorem}
To recover Theorem~\ref{thm:boolean-sse}, one simply applies the Hypercontractive Inequality with $f = 1_A$, $g = 1_B$ and optimizes the choice of $r, s$.  (We mention that this deduction was first noted, in its ``reverse'' form, by Mossel et al.~\cite{MOR+06}.) The Gaussian analogue of the Boolean Hypercontractive Inequality also holds; indeed, the traditional proof of it (say, in~\cite{Jan97}) involves first proving the Boolean inequality and then applying the Central Limit Theorem.

Another interpretation of the Hypercontractive Inequality is as a ``generalized \Holder's inequality''.  In fact, its $\rho = 1$ case (corresponding to $\by \equiv \bx$) is \emph{identical} to \Holder's inequality (since the hypothesis $\sqrt{rs} = 1$ is identical to $(1+s)' = 1+r$).  The Hypercontractive Inequality shows that as $\bx$ and $\by$ become less and less correlated, one can put smaller and smaller norms of~$f$ and~$g$ on the right-hand side.  (In the ultimate case of $\rho = 0$, meaning $\bx$ and $\by$ are independent, one gets the trivial inequality $\E[f(\bx)g(\by)] \leq \|f\|_1 \|g\|_1$.)

Speaking of \Holder's inequality, we should mention that it can be used to show that Theorem~\ref{thm:hypercon} is equivalent to the following more traditional formulation of the Hypercontractive Inequality:
\begin{equation}                                \label{eqn:holder-hypercon}
    \text{For $f \btR$, $1 \leq p \leq q \leq \infty$: }\ \  \|\T_\rho f\|_q \leq \|f\|_p  \text{ provided $0 \leq \rho \leq \sqrt{\tfrac{p-1}{q-1}}$.}
\end{equation}
Writing $p = 1+r$, $q = 1+1/s$, one uses the fact that $\|\T_\rho f\|_q = \sup\{\E[g \cdot \T_\rho f] : {\|g\|_{q'} = 1}\}$, and that the quantity inside the $\sup$ is the same as the left-hand side in Theorem~\ref{thm:hypercon}. Here we see an explanation for the name of the inequality --- it shows that $\T_\rho$ is not just a contraction in~$L^p$ but in fact is a ``hypercontraction'' from $L^p$ to $L^q$.  In this formulation, the inequality can be viewed as quantifying the ``smoothing'' effect of the~$\T_\rho$ operator. By virtue of Fact~\ref{fact:Trho-fourier} one can use this formulation to show that low-degree polynomials of independent $\pm 1$ random variables are ``reasonable'', in the sense that their high norms are comparable to their $2$-norm.  However we won't pursue this interpretation any further here.

A wonderful fact about the Boolean Hypercontractive Inequality is that the $n = 1$ case implies the general~$n$ case by induction.  Indeed, for the two-function form given in Theorem~\ref{thm:hypercon}, the induction is almost trivial.  If $(\bx, \by)$ are $\rho$-correlated and we write $\bx = (\bx_1, \bx')$ for $\bx' \in \bits^{n-1}$ (and similarly for~$\by$), then
\[
    \E[f(\bx) g(\by)] = \E_{(\bx_1, \by_1)} \E_{(\bx', \by')}[f_{\bx_1}(\bx') g_{\by_1}(\by')] \leq \E_{(\bx_1, \by_1)} [\|f_{\bx_1}\|_{1+r} \|g_{\by_1}\|_{1+s}],
\]
by induction, where $f_{x_1}$ denotes the restriction of $f$ gotten by fixing the first coordinate to be~$x_1$ (and similarly for $g_{y_1}$).  Then defining the $1$-bit functions $F(x_1) = \|f_{x_1}\|_{1+r}$ and $G(y_1) = \|g_{y_1}\|_{1+s}$ we have
\[
     \E_{(\bx_1, \by_1)} [\|f_{\bx_1}\|_{1+r} \|g_{\by_1}\|_{1+s}] =      \E_{(\bx_1, \by_1)} [F(\bx_1)G(\by_1)] \leq \|F\|_{1+r} \|G\|_{1+s} = \|f\|_{1+r} \|g\|_{1+s},
\]
where we used the $n = 1$ case of the Hypercontractive Inequality.

Thus to prove all of the Boolean and Gaussian Hypercontractivity and Small-Set Expansion theorems, it suffices to prove the $n = 1$ case of the Boolean Hypercontractive Inequality. In fact, by the \Holder trick we just need to prove~\eqref{eqn:holder-hypercon} in the case $n = 1$. It's also easy to show that we can assume $f \co \bits \to \R$ is nonnegative, and by homogeneity we can also assume $f$ has mean~$1$.  Thus everything boils down to proving the following: If $0 \leq \rho \leq \sqrt{\frac{p-1}{q-1}} \leq 1$ and $0 \leq \delta \leq 1$ then
\begin{equation}                                \label{eqn:two-point}
    \left(\tfrac12 (1+ \rho \delta)^q +     \tfrac12 (1- \rho \delta)^q\right)^{1/q} \leq     \left(\tfrac12 (1+ \delta)^p +     \tfrac12 (1-  \delta)^p\right)^{1/p}.
\end{equation}
Note that if we think of~$\delta$ as very small and perform a Taylor expansion, the above becomes
\[
    1 + \tfrac12 \rho^2 (q-1)\delta^2 + \cdots \leq     1 + \tfrac12 (p-1)\delta^2 + \cdots.
\]
This shows that the $\rho \leq \sqrt{\frac{p-1}{q-1}}$ condition is necessary, and also that it's  ``essentially'' sufficient assuming~$\delta$ is small. However, we need to actually verify~\eqref{eqn:two-point} for all $0 \leq \delta \leq 1$.  For some simple values of~$p$ and~$q$, this is easy.  For example, if $p = 2$ and $q = 4$, establishing~\eqref{eqn:two-point} amounts to noting that $1 + 2\delta^2 + \frac19 \delta^4 \leq 1 + 2\delta^2 + \delta^4$.  This is already enough to prove, say, the Boolean Small-Set Expansion statement~\eqref{eqn:b-sse} with parameter $\rho = \frac13$.  On the other hand, establishing~\eqref{eqn:two-point} for all $p, q$ and all~$\delta$ is a little bit painful (albeit elementary).  In the next section, we'll see a similar problem where this pain can be circumvented.

\section{Bobkov's Inequality and Gaussian Isoperimetry}     \label{sec:bobkov}

Let's now look at a different special case of Borell's Isoperimetric Inequality, namely the case where $B = A$ and  $\rho \to 1^{-}$.  Using the rotation sensitivity definition from Corollary~\ref{cor:easy-borell}, Borell's inequality tells us that if $A \subseteq \R^d$, and $H \subseteq \R^d$ is a halfspace of the same Gaussian volume, then $\RS_A(\delta) \geq \RS_H(\delta)$.  Since we also have $\RS_A(0) = \RS_H(0) = 0$, it follows that $\RS_A'(0^+) \geq \RS_H'(0^+)$. (It can be shown that this derivative $\RS'_A(0^+)$ is always well-defined, though it may be~$\infty$.)  As we'll explain shortly, the derivative $\RS_A'(0^+)$ has a very simple meaning; up to a factor of~$\sqrt{\frac{\pi}{2}}$, it is the \emph{Gaussian surface area} of the set~$A$.  Thus Borell's Isoperimetric Inequality implies the following well-known result:
\begin{theorem}[``Gaussian Isoperimetric Inequality'']                                     \label{thm:giso}
    Let $A \subseteq \R^d$ have Gaussian volume $\gamma_d(A) = \alpha$, and let $H \subseteq \R^d$ be any halfspace with $\gamma_d(H) = \alpha$.  Then
    \begin{equation}                                \label{eqn:giso}
        \surf_d(A) \geq \surf_d(H).
    \end{equation}
\end{theorem}
Here we are using the following definition:
\begin{definition}                                  \label{def:giso}
    The \emph{Gaussian surface area} of $A \subseteq \R^d$ is
    \[
        \surf_d(A) = \sqrt{\frac{\pi}{2}} \cdot \RS_A'(0^+) = \lim_{\delta \to 0^+}\frac{\gamma_d((\bdry A)^{+\delta/2})}{\delta} = \E_{\bz \sim \gamma_d}[\|\grad 1_A(\bz)\|] = \int_{\bdry A} \vphi(x)\,dx.
    \]
    The first equation may be taken as the definition, and the remaining equations hold assuming $A$ is ``nice enough'' (for technical details, see~\cite{AMMP10,AFR13,Hin10,AF11,MNP12,AFR13}).
\end{definition}
To get a feel for the definition, let's ``heuristically justify'' the second equality above, which relates the derivative of rotation sensitivity to the more natural-looking Gaussian Minkowski content of $\bdry A$.  We can think of
\begin{equation}                                \label{eqn:rsa0}
    \RS_A'(0^+) = \frac{\RS_A(\delta)}{\delta} = \frac{1}{\delta} \Pr_{\substack{(\bz, \bz') \textnormal{ $\cos \delta$-correlated} \\ \textnormal{$d$-dimensional Gaussian pair}}}[1_A(\bz) \neq 1_A(\bz')]
\end{equation}
for ``infinitesimal''~$\delta$.  The last expression here can be thought of as the probability that the line segment $\bell$ joining $\bz, \bz'$ crosses $\bdry A$.  Now for infinitesimal~$\delta$ we have $\cos \delta \approx 1$ and $\sin \delta \approx \delta$; thus the distribution of $(\bz, \bz')$ is essentially that of $(\bg, \bg + \delta \bg')$ for $\bg$, $\bg'$ independent $d$-dimensional Gaussians.  When $\bg$ lands near~$\bdry A$, the length of the segment $\bell$ in the direction of the nearby unit normal~$\bv$ to~$\bdry A$ will have expectation $\E[|\la \delta \bg', \bv \ra|] = \delta \E[|\normal(0,1)|] = \sqrt{2/\pi} \cdot \delta$.  Thus~\eqref{eqn:rsa0} should essentially be~$\sqrt{2/\pi} \cdot \delta \cdot \gamma_d(\{z : \dist(z, \bdry A) < \delta\})$, completing the heuristic justification of the second inequality in Definition~\ref{def:giso}.

Incidentally, it's easy to see that the Gaussian surface area of the one-dimensional halfspace $(-\infty, a] \subseteq \R$ is $\vphi(a)$; thus we can give an explicit formula for the right-hand side of~\eqref{eqn:giso}:
\begin{fact}                                        \label{fact:giso}
    The right-hand side of~\eqref{eqn:giso} is the \emph{Gaussian isoperimetric function},
    \[
        \giso(\alpha) = \vphi \circ \Phi^{-1}(\alpha) \in [0, \tfrac{1}{\sqrt{2\pi}}].
    \]
    A remark: One  easily checks that it satisfies the differential equation ${\giso \giso'' + 1 = 0}$, with boundary conditions $\giso(0) = \giso(1) = 0$.
\end{fact}

The Gaussian Isoperimetric Inequality was originally independently proven by Borell~\cite{Bor75} and by Sudakov and Tsirel'son~\cite{ST78}.  Both proofs deduced it via Poincar\'e's limit from L\'evy's Spherical Isoperimetric Inequality~\cite{Lev22,Sch48}.  (This is the statement that the fixed-volume subsets of a sphere's surface which minimize perimeter are caps --- i.e., intersections of the sphere with a halfspace.) Ehrhard~\cite{Ehr83} subsequently developed his Gaussian symmetrization method to give  a different proof.  In 1997, Bobkov gave a surprising new proof by the same technique we saw in the last section: establishing a functional Boolean analogue by induction.  We'll now outline this proof.

We start with the following equivalent functional form of the Gaussian Isoperimetric Inequality (first noted by Ehrhard~\cite{Ehr84}):  For locally Lipschitz ${f \co \R^d \to [0,1]}$,
\begin{equation}                                            \label{eqn:gaussian-bobkov}
    \giso(\E[f(\bz)]) \leq \E[\|(\grad f(\bz), \giso(f(\bz)))\|_2],
\end{equation}
where $\bz \sim \gamma_d$ and $\|\cdot\|_2$ denotes the usual Euclidean norm in $d+1$ dimensions.  The Gaussian Isoperimetric Inequality for~$A$ can be deduced by taking $f = 1_A$; conversely, inequality~\eqref{eqn:gaussian-bobkov} can be deduced from the Gaussian Isoperimetric Inequality by taking $A = \{(x, a) : f(x) \geq \Phi(a)\} \subseteq \R^{d+1}$.  In turn, Bobkov showed that the above inequality can be deduced (by the usual Central Limit Theorem argument) from the analogous Boolean inequality:
\begin{theorem}[``Bobkov's Inequality'']                                     \label{thm:bobkov}
    For any $f \co \bn \to [0,1]$,
    \[
         \giso(\E[f]) \leq \E[\|(\grad f, \giso(f))\|_2].
    \]
    Here the expectation is with respect to the uniform distribution on $\bn$, and $\grad f = (\D_1 f, \dots, \D_n f)$.
\end{theorem}
Just as with the Hypercontractive Inequality, this inequality has the property that the $n = 1$ case implies the general~$n$ case by a fairly easy induction.  Indeed, this induction uses no special property of $\giso$ or the $2$-norm:
\begin{fact}                                        \label{fact:bobkov-induction}
    Let $J \co [0,1] \to \R^{\geq 0}$, and let $\| \cdot \|$ denote a fixed $L^p$-norm.  Consider, for $f \co \bn \to [0,1]$, the following inequality:
    \begin{equation}                            \label{eqn:bobkov1}
         J(\E[f]) \leq \E[\|(\grad f, J(f))\|].
    \end{equation}
    If this inequality holds for $n = 1$ then it holds for general~$n$.
\end{fact}
Now given a norm $\| \cdot \|$ we can seek the ``largest'' function~$J$ for which~\eqref{eqn:bobkov1} holds when $n = 1$.  As an aside, for the $1$-norm $\| \cdot \|_1$ we may take $J(\alpha) = \alpha \log_2(1/\alpha)$, and this yields a form of the classic Edge Isoperimetric Inequality for the discrete cube~\cite{Har64}, sharp for all $\alpha = 2^{-k}, k \in \Z^+$.  Returning to Bobkov's Inequality, the $n = 1$ case we need to verify is that
\begin{equation}                                \label{eqn:bobkov-key}
    J(\alpha) \leq \half \sqrt{\delta^2 + J(\alpha + \delta)^2} + \half \sqrt{\delta^2 + J(\alpha - \delta)^2}
\end{equation}
when $J = \giso$ and $\alpha \pm \delta \in [0,1]$.  Bobkov used some (elementary) labor to show that this inequality indeed holds when $J = \Giso$.  To see how the Gaussian isoperimetric function arises, we Taylor-expand the right-hand side in~$\delta$, getting:
\begin{equation}                                \label{eqn:bobkov-taylor}
    J(\alpha) + \frac{1}{2J(\alpha)}(J(\alpha) J''(\alpha)+1) \delta^2 \pm O(\delta^4).
\end{equation}
Thus if take $J = \Giso$, which satisfies $\Giso \Giso'' + 1 = 0$, then the needed inequality~\eqref{eqn:bobkov-key} will at least be satisfied ``for small $\delta$, up to an additive $o(\delta^2)$''.

Perhaps surprisingly, this is enough to deduce that~\eqref{eqn:bobkov-key} holds exactly, for all~$\delta$.  This was (in a sense) first established by Barthe and Maurey, who used stochastic calculus and It\^o's Formula to prove that~\eqref{eqn:bobkov-key} holds with $J = \Giso$.  Let us present here a sketch of an elementary, discrete version of Barthe--Maurey argument.

We wish to show that Theorem~\ref{thm:bobkov} holds in the $n = 1$ case; say, for the function $f(\by) = \alpha + \beta \by$, where $\by \sim \bits$. Let's take a random walk on the line, starting from~$0$, with independent increments $\bx_1, \bx_2, \bx_3, \dots$ of $\pm \delta$, and stopping when the walk reaches $\pm 1$ (we assume $1/\delta \in \Z^+$).  We let $\by \in \{-1,1\}$ be the stopping point of this walk (which is equally likely to be~$\pm 1$).  Now proving Bobkov's inequality for $f(\by) = \alpha + \beta(\bx_1 + \bx_2 + \bx_3 + \cdots)$ can be reduced to proving Bobkov's inequality just for $f(\bx_1) = \alpha + \beta \bx_1$, essentially by the same easy induction used to derive Theorem~\ref{thm:bobkov} from its $n = 1$ case.  This puts us back in the same position as before: we need to show that
\[
    \giso(\alpha) \leq \half \sqrt{(\beta \delta)^2 + \giso(\alpha + \beta\delta)^2} + \half \sqrt{(\beta\delta)^2 + \giso(\alpha - \beta\delta)^2}.
\]
However  we now have the advantage that the quantity $\beta \delta$ is indeed ``small''; we can make it as small as we please.  By the Taylor expansion~\eqref{eqn:bobkov-taylor}, the above inequality indeed holds up to an additive $o(\delta^2)$ error.  Furthermore, if we simply let this error accumulate in the induction, it costs us almost nothing. It's well known and simple that if $\bT$ is the number of steps the random walk takes before stopping, then $\E[\bT] = 1/\delta^2$.  Thus we can afford to let an~$o(\delta^2)$ error accumulate for $1/\delta^2$ steps, since~$\delta$ can be made arbitrarily small.

The Barthe--Maurey version of the above argument replaces the random walk with Brownian motion; this is arguably more elegant, but less elementary.  An amusing aspect of all this is the following:  We first saw in Section~\ref{sec:borell} that statements about Gaussian geometry can be proven by ``simulating'' Gaussian random variables by sums of many random $\pm 1$ bits (scaled down);  the above argument shows that it can also be effective to simulate a single $\pm 1$ random bit by the sum of many small Gaussians (i.e., with Brownian motion).

We end this section by mentioning that Bobkov's approach to the Gaussian Isoperimetric Inequality inspired Bakry and Ledoux~\cite{BL96b,Led98} to give a ``semigroup proof'' of the Gaussian version of Bobkov's inequality~\eqref{eqn:gaussian-bobkov} (\'a~la~\cite{BE85a,Led94}).  Specifically, if one defines
\[
    F(\rho) = \E_{\gamma_d}[\|(\grad \U_\rho f,\ \giso(\U_\rho f))\|_2],
\]
then they showed that~$F$ is a nondecreasing function of $\rho \in [0,1]$ just by differentiation (though the computations are a bit cumbersome). This immediately implies~\eqref{eqn:gaussian-bobkov} by taking $\rho = 0, 1$.  Mossel and Neeman~\cite{MN12} proved the more general Borell Isoperimetric Inequality using a very similar semigroup technique, and Ledoux~\cite{Led13} generalized their methodology to include the Hypercontractive Inequality, Brascamp--Lieb inequalities, and some forms of the Slepian inequalities.  However, it was by returning to discrete methods --- i.e., proving a statement about Boolean functions by induction --- that De, Mossel, and Neeman~\cite{DMN13} were able to simultaneously establish the Majority Is Stablest Theorem and Borell's theorem.

\section{The De--Mossel--Neeman proof of the MIST}                      \label{sec:DMN}

Mossel and Neeman actually proved the following functional version of Borell's Isoperimetric Inequality:
\begin{theorem}                                     \label{thm:mn}
    Fix $0 < \rho < 1$ and let $f, g \co \R^d \to [0,1]$.  Then if $(\bz, \bz')$ is a $\rho$-correlated $d$-dimensional Gaussian pair,
    \begin{equation}                                \label{eqn:mn}
        \E[\Lambda_{\rho}(f(\bz),g(\bz'))] \leq \Lambda_{\rho}(\E[f(\bz)], \E[g(\bz')]).
    \end{equation}
    (If $-1 < \rho < 0$ then the inequality is reversed.)
\end{theorem}
This is equivalent to Borell's inequality in the same way that~\eqref{eqn:gaussian-bobkov} is equivalent to the Gaussian Isoperimetric Inequality (note in particular that $\Lambda_\rho(\alpha,\beta) = \alpha \beta$ when $\alpha,\beta \in \{0,1\}$).  This inequality also has the property that the general-$d$ case follows from the $d = 1$ case by a completely trivial induction, using no special property of $\Lambda_\rho$ or the Gaussian distribution; it only uses that the~$d$ pairs $(\bz_i, \bz'_i)$ are independent.  In particular, \emph{if}~\eqref{eqn:mn} were to hold for one-bit functions $f, g \co \{-1,1\} \to [0,1]$ then we could deduce it for general $f, g \co \{-1,1\}^n \to [0,1]$ by induction, then for Gaussian $f, g \co \R \to [0,1]$ by the Central Limit Theorem, and finally for Gaussian $f, g \co \R^d \to [0,1]$ by induction again.  Unfortunately, the inequality~\eqref{eqn:mn} does \emph{not} hold for $f, g \co \bits \to [0,1]$.  It's clear that it can't, because otherwise we would obtain the Majority Is Stablest Theorem with no hypothesis about small influences (which is false).  Indeed, the ``dictator'' functions $f, g \co \bits \to [0,1]$, $f(x) = g(x) = \half + \half x$ provide an immediate counterexample; inequality~\eqref{eqn:mn} becomes the false statement $\frac14 + \frac14\rho \leq \frac12 - \frac{1}{2\pi}\arccos \rho$.

Nevertheless, as noted by De, Mossel, and Neeman~\cite{DMN13} we are back in the situation wherein~\eqref{eqn:mn} ``essentially'' holds for one-bit functions ``with small influences''; i.e., for $f(x) = \alpha + \delta_1 x$, $g(x) = \beta + \delta_2 x$ with $\delta_1, \delta_2$ ``small''.  To see this, Taylor-expand the left-hand side of~\eqref{eqn:mn} around~$(\alpha,\beta)$:
\begin{align}                                   \label{eqn:dmn-1}
     \E_{\substack{(\bx, \bx') \\ \textnormal{$\rho$-correlated}}}[\Lambda_{\rho}(f(\bx),g(\bx'))]
    &= \Lambda_\rho(\alpha,\beta) + \E[\delta_1 \bx \cdot \D_1 \Lambda_\rho(\alpha,\beta)] + \E[\delta_2 \bx' \cdot \D_2 \Lambda_\rho(\alpha,\beta)] \nonumber\\
    &+\ \E\left[\begin{bmatrix} \delta_1 \bx & \delta_2 \bx' \end{bmatrix} \cdot H\Lambda_\rho(\alpha,\beta)  \cdot \begin{bmatrix} \delta_1 \bx \\ \delta_2 \bx' \end{bmatrix}\right] + \cdots
\end{align}
(Here $H \Lambda_\rho$ denotes the Hessian of $\Lambda_\rho$.) The first term here matches the right-hand side of~\eqref{eqn:mn}.  The second and third terms vanish, since $\E[\bx] = \E[\bx'] = 0$.  Finally, since $\E[\bx\bx'] = \rho$ the fourth term is
\begin{equation}                                    \label{eqn:dmn-4}
    \begin{bmatrix} \delta_1 & \delta_2 \end{bmatrix} \cdot H_\rho \Lambda_\rho(\alpha,\beta)  \cdot \begin{bmatrix} \delta_1 \\ \delta_2 \end{bmatrix}, \quad \text{where the notation }     H_\rho F \text{ means } \begin{bmatrix} 1 &  \rho \\ \rho & 1 \end{bmatrix} \circ HF.
\end{equation}
One can show by a relatively short calculation that $\det(H_\rho \Lambda_\rho)$ is identically~$0$ and that the diagonal entries of $H_\rho \Lambda_\rho$ always have opposite sign to~$\rho$.  Thus for $0 < \rho < 1$, the matrix $H_\rho \Lambda_\rho$ is everywhere negative semidefinite and hence~\eqref{eqn:dmn-4} is always nonpositive.  (The reverse happens for $0 < \rho < 1$.) Ledoux~\cite{Led13}
introduced the terminology \emph{$\rho$-concavity of~$F$} for the condition $H_\rho F \preccurlyeq 0$.

It follows that~\eqref{eqn:mn} indeed holds for one-bit Boolean functions $f, g$, up to the ``cubic error term'' elided in~\eqref{eqn:dmn-1}.  If one now does the induction while keeping these cubic error terms around, the result is the following:
\begin{theorem}[``De--Mossel--Neeman Theorem'']             \label{thm:dmn}
    Fix $0 < \rho < 1$ and any small $\eta > 0$.  Then for $f, g \co \bits^n \to [\eta, 1-\eta]$,
    \begin{equation}                                \label{eqn:mn2}
        \E_{\substack{(\bx, \by) \\ \textnormal{$\rho$-correlated}}}[\Lambda_{\rho}(f(\bx),g(\by))] \leq \Lambda_{\rho}(\E[f(\bx)], \E[g(\by)]) + O_{\rho, \eta}(1) \cdot \sum_{i=1}^n  (\|\mathrm{d}_i f\|_3^3 + \|\mathrm{d}_i g\|_3^3),
    \end{equation}
    where $\mathrm{d}_i h$ denotes the \emph{$i$th martingale difference} for~$h$,
    \[
        (\bx_1, \dots, \bx_i) \mapsto \E[h \mid \bx_1, \dots, \bx_i] - \E[h \mid \bx_1, \dots, \bx_{i-1}].
    \]
    (For $-1 < \rho < 0$, the inequality~\eqref{eqn:mn2} is reversed.)
\end{theorem}
With this theorem in hand, Borell's Isoperimetric Inequality for Gaussian functions $f, g \co \R \to [\eta,1-\eta]$ is easily deduced by the standard Central Limit Theorem argument: one only needs to check that the cubic error term is~$O(\frac{1}{\sqrt{n}})$, and~$n$ may be taken arbitrarily large.  Then one immediately deduces the full Borell theorem by taking $\eta \to 0$ and doing another induction on the Gaussian dimension~$d$.  On top of this, De, Mossel, and Neeman showed how to deduce Majority Is Stablest from Theorem~\ref{thm:dmn}, using a small collection of analytical tricks appearing in the original proof.  The key trick is to use hypercontractivity to bound $\|\mathrm{d}_i f\|_3^3$ in terms of
\[
    (\|\D_i f\|_2^2)^{1+\delta} = \Inf_i[f]^{1+\delta}
\]
for some small $\delta \approx \frac{\log \log(1/\eps)}{\log(1/\eps)} > 0$. The fact that we get the nontrivial extra factor $\Inf_i[f]^{\delta}$, which is at most $\eps^{\delta} \approx \frac{1}{\log(1/\eps)}$ by assumption, is the key to finishing the proof.

\section{Conclusions: proof complexity}                               \label{sec:conclusion}

As mentioned, there are two known proofs of the Majority Is Stablest Theorem: the original one, which used the Invariance Principle to reduce the problem to Borell's Isoperimetric Inequality; and, the elegant one due to De, Mossel, and Neeman, which is a completely ``discrete proof'', as befits a purely discrete problem like Majority Is Stablest.  Esthetics is not the only merit of the latter proof, however; as we describe in this section, the fact that the De--Mossel--Neeman proof is simpler and more discrete leads to new technical results concerning the computational complexity of Max-Cut.

Regarding Max-Cut, let's consider the closely related problem of \emph{certifying} that  a given graph has no large cut.  As we saw in Section~\ref{sec:max-cut}, for any graph~$G$ we can use semidefinite programming to efficiently compute a value $\beta = \SDPOpt(G)$ such that the maximum cut in $G$ satisfies $\Opt(G) \leq \beta$.  We think of this algorithm as producing a \emph{proof} of the statement ``$\Opt(G) \leq \beta$''.  Furthermore, the (analysis of the) Goemans--Williamson algorithm implies that the bound found by this algorithm is fairly good; whenever $G$ truly satisfies $\Opt(G) \leq \frac{\theta}{\pi}$ (for $\theta \in [\thetagw, \pi]$), we will efficiently obtain a proof of ``$\Opt(G) \leq \half - \half \cos \theta$''.  For example, if $\Opt(G) \leq \frac34$ then there is an efficiently-obtainable ``SDP proof'' of the statement ``$\Opt(G) \leq \half + \frac{1}{2\sqrt{2}} \approx .854$''.

Assuming the Unique Games Conjecture (and $\PTIME \neq \NP$), the works~\cite{KKMO04,MOO05} imply that there is no efficient algorithm that can in general find better proofs; e.g., that can certify ``$\Opt(G) \leq .853$'' whenever $\Opt(G) \leq \frac34$.  In fact, under the additional standard assumption of $\coNP \neq \NP$, the implication is simply that no short proofs \emph{exist}; i.e., there are infinite families of graphs $G = (V,E)$ with $\Opt(G) \leq \frac34$ but no $\poly(|V|)$-length proof of the statement ``$\Opt(G) \leq .853$'' (say, in some textbook formalization of mathematical reasoning).  In other words:

\paragraph{Unique Games \& $\PTIME \neq \NP$ Prediction about Max-Cut:} Let $\theta \in [\thetagw, \pi]$ and $\delta > 0$.  There is no polynomial-time algorithm that, given a Max-Cut instance~$G$ with $\Opt(G) \leq \frac{\theta}{\pi}$, outputs a proof of ``$\Opt(G) \leq \half - \half \cos \theta - \delta$''.

\paragraph{Unique Games \& $\coNP \neq \NP$ Prediction about Max-Cut:} In fact, there are infinitely many graphs~$G$ with $\Opt(G) \leq \frac{\theta}{\pi}$, yet for which no polynomial-length proof of ``$\Opt(G) \leq \half - \half \cos \theta - \delta$'' exists.\\

As mentioned, the Unique Games Conjecture is quite contentious, so it's important to seek additional evidence concerning the above predictions.  For example, to support the first prediction one should at a minimum show that the semidefinite program~\eqref{eqn:gw-sdp} fails to provide such proofs.  That is, one should find graphs~$G$ with $\Opt(G) \leq \frac{\theta}{\pi}$ yet $\SDPOpt(G) \geq \half - \half \cos \theta$.  Such graphs are called \emph{SDP integrality gap instances}, as they exhibit a large gap between their true optimal Max-Cut and the upper-bound certified by the SDP.  Borell's Isoperimetric Inequality precisely provides such graphs, at least if ``weighted continuous graphs'' are allowed: One takes the ``graph'' $G$ whose vertex set is~$\R^d$ and whose ``edge measure'' is given by choosing a $(\cos \theta)$-correlated pair of Gaussians.  The fact that $\Opt(G) \leq \frac{\theta}{\pi}$ is immediate from Borell's Theorem~\ref{thm:gaussian-mist}; further, it's not hard to show (using the idea of Remark~\ref{rem:sphere-idea}) that choosing $\vec{U}(v) = v/\sqrt{d}$ in~\eqref{eqn:gw-sdp} establishes $\SDPOpt(G) \geq \half - \half \cos \theta - o_d(1)$.  These facts were essentially established originally by Feige and Schechtman~\cite{FS02}, who also showed how to discretize the construction so as to provide finite integrality gap graphs.

(Incidentally, we may now explain that the Raghavendra Theory mentioned at the end of Section~\ref{sec:max-cut} significantly generalizes the work of Khot et al.~\cite{KKMO04} by showing how to transform an SDP integrality gap instance for \emph{any} CSP into a matching computational hardness-of-approximation result, assuming the Unique Games  Conjecture.)

Although the semidefinite program~\eqref{eqn:gw-sdp} fails to certify $\Opt(G) \leq \frac{\theta}{\pi}$ for the ``correlated Gaussian graphs'' described above, a great deal of recent research  has gone into developing stronger ``proof systems'' for reasoning about Max-Cut and other CSPs.  (See, e.g.,~\cite{Geo10} for a survey.) Actually, until recently this research was viewed not in terms of proof complexity but in terms of analyzing ``tighter'' SDP relaxations that can still be solved efficiently.  For example, one can still solve the optimization problem~\eqref{eqn:gw-sdp} in polynomial time with the following ``triangle inequality'' constraint added in:
\[
    \la U(v), U(w) \ra + \la U(w), U(x) \ra - \la U(v), U(x)\ra \leq 1 \quad \forall v,w,x \in V.
\]
Note that with this additional constraint we still have $\Opt(G) \leq \SDPOpt(G)$ for all~$G$, because the constraint is satisfied by any genuine bipartition $U \co V \to \bits$. As noted by Feige and Schechtman~\cite{FS02}, adding this constraint gives a certification better than ``$\Opt(G) \leq \half - \half \cos \theta$'' for the Gaussian correlation graphs, though it's not clear by how much.

Although this stronger ``SDP + triangle inequality'' proof system does better on Gaussian correlation graphs, a breakthrough work of Khot and Vishnoi~\cite{KV05} showed that it still suffers from the same integrality gap for a different infinite family of graphs.  In other words, even when the SDP includes the triangle inequalities, these \emph{Khot--Vishnoi graphs}~$G = (V,E)$ have $\SDPOpt(G) \geq \half - \half \cos \theta$ yet $\Opt(G) \leq \frac{\theta}{\pi} + o_{|V|}(1)$.  The second fact, the upper bound on the true Max-Cut value, relies directly on the Majority Is Stablest Theorem.  Subsequent works~\cite{KS09,RS09b} significantly generalized this result by showing that even much tighter ``SDP hierarchies'' still fail to certify anything better than ``$\Opt(G) \leq \half - \half \cos \theta$'' for the Khot--Vishnoi graphs~$G$.  This could be considered additional evidence in favor of the Unique Games \& $\PTIME \neq \NP$ Prediction concerning Max-Cut.

A recent work by Barak et al.~\cite{BBH+12} cast some doubt on this prediction, however.  Their work showed that the especially strong ``Lasserre/Parrilo SDP hierarchy''~\cite{Sho87,Las00,Par00} succeeds in finding some good CSP bounds which weaker SDP hierarchies are unable to obtain.  Specifically, they showed it provides good upper bounds on the optimal value of the Khot--Vishnoi ``Unique Games instances'' (which are, in some sense, subcomponents of the Khot--Vishnoi Max-Cut graphs).  Subsequent work of O'Donnell and Zhou~\cite{OZ13} further emphasized the equivalence of the Lasserre/Parrilo SDP hierarchy and the \emph{Sum-of-Squares (SOS) proof system}, invented by Grigroriev and Vorobjov~\cite{GV01}. In the context of the Max-Cut CSP, this proof system (inspired by Hilbert's~17th Problem~\cite{Hil02} and the \emph{Positivstellensatz} of Krivine~\cite{Kri64} and Stengle~\cite{Ste73}) seeks to establish the statement ``$\Opt(G) \leq \beta$'' for a graph $G = (V,E)$ by expressing
\begin{equation}                                    \label{eqn:SOS}
    \beta - \left(\avg_{(v,w) \in E} \half - \half X_v X_w\right) = \sum_{i=1}^s P_i^2 \quad \text{within the ring } \R[(X_v)_{v \in V}]/(X_v^2 - 1)_{v \in V},
\end{equation}
for some formal polynomials $P_1, \dots, P_s$ of degree at most some constant~$C$.  Somewhat remarkably, there is an efficient ($|V|^{O(C)}$-time) algorithm for finding such $P_i$'s whenever they exist.

As mentioned, for the Khot--Vishnoi Max-Cut graphs~$G$, the fact that $\Opt(G) \leq \frac{\theta}{\pi} + o(1)$ follows directly from the Majority Is Stablest Theorem.  To show that the SOS proof system can also certify this fact (thereby casting some doubt on the Unique Games \& $\PTIME \neq \NP$ Prediction about Max-Cut), one needs to show that not only is the Majority Is Stablest Theorem true, but that it can be proved within the extremely constrained  SOS proof system, \'a~la~\eqref{eqn:SOS}.  The original proof of the Majority Is Stablest Theorem was quite complicated, using the Invariance Principle from~\cite{MOO10} to reduce Borell's Isoperimetric Inequality, and then relying on the known geometric proofs~\cite{Bor85,Bec92} of the latter.  The prospect for converting this proof into an SOS format seemed quite daunting (although a partial result was established in~\cite{OZ13}, showing that the SOS proof system can establish ``$\Opt(G) \leq \half - \frac{\cos\theta}{\pi} - (\half - \frac{1}{\pi})\cos^3 \theta$'').  However, the simplicity and discrete nature of the new De--Mossel--Neeman proof of the Majority Is Stablest Theorem allowed them to show that the SOS proof system \emph{can} establish the truth about the Khot--Vishnoi graphs, $\Opt(G) \leq \half - \half \cos \theta + o(1)$.

It is to be hoped that this result can be extended to the entire Raghavendra Theory, thereby showing that the SOS proof system can certify the optimal value of the analogue of the Khot--Vishnoi instances for \emph{all} CSPs.  However as the Raghavendra Theory still relies on the Invariance Principle, whether or not this is possible is unclear.

Finally, in light of the De--Mossel-Neeman result, the following interesting question is open:  Are there (infinite families of) instances of the Max-Cut problem~$G$ such that $\Opt(G) \leq \frac{\theta}{\pi}$, yet such that any mathematical proof of this statement is so complicated that the SOS proof system cannot establish anything better than ``$\Opt(G) \leq \half - \half \cos \theta$''?  If such graphs were found, this might tilt the weight of evidence back in favor of the Unique Games \& $\PTIME \neq \NP$ Prediction. Of course, if human mathematicians explicitly construct the proof of $\Opt(G) \leq \frac{\theta}{\pi}$, presumably it will have polynomial length, and therefore not provide any evidence in favor of the Unique Games \& $\coNP \neq \NP$ Prediction.  To provide evidence for this stronger prediction, one presumably needs to give a \emph{probabilistic} construction of graphs~$G$ such that both of the following happen with high probability: (i)~$\Opt(G) \leq \frac{\theta}{\pi}$; and, (ii)~there is no polynomial-length proof even of ``$\Opt(G) \leq \half - \half \cos \theta$''.

\frenchspacing


\vspace{-.04in}

\begin{thebibliography}{10}

\bibitem{AF11}
L.~Ambrosio and A.~Figalli.
\newblock Surface measures and convergence of the {O}rnstein--{U}hlenbeck
  semigroup in {W}iener spaces.
\newblock {\em Annales de la facult\'{e} des sciences de {T}oulouse
  {M}ath\'{e}matiques (s\'{e}rie 6)}, 20(2):407--438, 2011.

\bibitem{AFR13}
L.~Ambrosio, A.~Figalli, and E.~Runa.
\newblock On sets of finite perimeter in {W}iener spaces: reduced boundary and
  convergence to halfspaces.
\newblock {\em Atti della Accademia Nazionale dei Lincei. Classe di Scienze
  Fisiche, Matematiche e Naturali. Rendiconti Lincei. Serie IX. Matematica e
  Applicazioni}, 24(1):111--122, 2013.

\bibitem{AMMP10}
L.~Ambrosio, M.~{Miranda~Jr.}, S.~Maniglia, and D.~Pallara.
\newblock {BV} functions in abstract {W}iener spaces.
\newblock {\em Journal of Functional Analysis}, 258(3):785--813, 2010.

\bibitem{ALM+98}
S.~Arora, C.~Lund, R.~Motwani, M.~Sudan, and M.~Szegedy.
\newblock Proof verification and the hardness of approximation problems.
\newblock {\em Journal of the ACM}, 45(3):501--555, 1998.

\bibitem{AS98}
S.~Arora and S.~Safra.
\newblock Probabilistic checking of proofs: {A} new characterization of {NP}.
\newblock {\em Journal of the ACM}, 45(1):70--122, 1998.

\bibitem{Arr50}
K.~Arrow.
\newblock A difficulty in the concept of social welfare.
\newblock {\em The Journal of Political Economy}, 58(4):328--346, 1950.

\bibitem{BT76}
A.~Baernstein and B.~Taylor.
\newblock Spherical rearrangements, subharmonic functions, and $*$-functions in
  $n$-space.
\newblock {\em Duke Mathematical Journal}, 43(2):245--268, 1976.

\bibitem{BE85a}
D.~Bakry and M.~{\'E}mery.
\newblock Diffusions hypercontractives.
\newblock In {\em S\'eminaire de {P}robabilit\'es, {XIX}}, volume 1123 of {\em
  Lecture Notes in Mathematics}, pages 177--206. Springer, Berlin, 1985.

\bibitem{BL96b}
D.~Bakry and M.~Ledoux.
\newblock L\'{e}vy--{G}romov's isoperimetric inequality for an infinite
  dimensional diffusion generator.
\newblock {\em Inventiones mathematicae}, 123(1):259--281, 1996.

\bibitem{Ban65}
J.~Banzhaf.
\newblock Weighted voting doesn't work: A mathematical analysis.
\newblock {\em Rutgers Law Review}, 19:317--343, 1965.

\bibitem{BBH+12}
B.~Barak, F.~Brand{\~a}o, A.~Harrow, J.~Kelner, D.~Steurer, and Y.~Zhou.
\newblock Hypercontractivity, sum-of-squares proofs, and their applications.
\newblock In {\em Proceedings of the 44th Annual ACM Symposium on Theory of
  Computing}, pages 307--326, 2012.

\bibitem{Bec92}
W.~Beckner.
\newblock Sobolev inequalities, the {Poisson} semigroup, and analysis on the
  sphere {$S^n$}.
\newblock {\em Proceedings of the National Academy of Sciences},
  89(11):4816--4819, 1992.

\bibitem{BGS98}
M.~Bellare, O.~Goldreich, and M.~Sudan.
\newblock Free bits, {PCP}s, and non-approximability -- towards tight results.
\newblock {\em SIAM Journal of Computing}, 27(3):804--915, 1998.

\bibitem{BO10}
E.~Blais and R.~O'Donnell.
\newblock Lower bounds for testing function isomorphism.
\newblock In {\em Proceedings of the 25th Annual IEEE Conference on
  Computational Complexity}, pages 235--246, 2010.

\bibitem{Bon70}
A.~Bonami.
\newblock {\'E}tude des coefficients {F}ourier des fonctions de {$L^{p}(G)$}.
\newblock {\em Annales de l'Institut Fourier}, 20(2):335--402, 1970.

\bibitem{Bor75}
C.~Borell.
\newblock The {B}runn--{M}inkowski inequality in {G}auss space.
\newblock {\em Inventiones Mathematicae}, 30(2):207--216, 1975.

\bibitem{Bor85}
C.~Borell.
\newblock Geometric bounds on the {O}rnstein--{U}hlenbeck velocity process.
\newblock {\em Probability Theory and Related Fields}, 70(1):1--13, 1985.

\bibitem{BGR09}
S.~Brams, W.~Gehrlein, and F.~Roberts, editors.
\newblock {\em The Mathematics of Preference, Choice and Order}.
\newblock Springer, 2009.

\bibitem{CL90}
E.~Carlen and M.~Loss.
\newblock Extremals of functionals with competing symmetries.
\newblock {\em Journal of Functional Analysis}, 88(2):437--456, 1990.

\bibitem{DMN13}
A.~De, E.~Mossel, and J.~Neeman.
\newblock Majority is {S}tablest : {D}iscrete and {SoS}.
\newblock In {\em Proceedings of the 45th Annual ACM Symposium on Theory of
  Computing}, 2013.

\bibitem{dC85}
N.~de~Condorcet.
\newblock {\em Essai sur l'application de l'analyse \`{a} la probabilit\'{e}
  des d\'{e}cisions rendues \`{a} la pluralit\'{e} des voix}.
\newblock Paris, de l'imprimerie royale, 1785.

\bibitem{DP93}
C.~Delorme and S.~Poljak.
\newblock Laplacian eigenvalues and the maximum cut problem.
\newblock {\em Mathematical Programming}, 62(1--3):557--574, 1993.

\bibitem{DFR08}
I.~Dinur, E.~Friedgut, and O.~Regev.
\newblock Independent sets in graph powers are almost contained in juntas.
\newblock {\em Geometric and Functional Analysis}, 18(1):77--97, 2008.

\bibitem{Ehr83}
A.~Ehrhard.
\newblock Sym{\'e}trisation dans l'espace de gauss.
\newblock {\em Mathematica Scandinavica}, 53:281--301, 1983.

\bibitem{Ehr84}
A.~Ehrhard.
\newblock In\'egalit\'es isop\'erim\'etriques et int\'egrales de {D}irichlet
  gaussiennes.
\newblock {\em Annales Scientifiques de l'\'Ecole Normale Sup\'erieure.
  Quatri\`eme S\'erie}, 17(2):317--332, 1984.

\bibitem{Eld13}
R.~Eldan.
\newblock A two-sided estimate for the {G}aussian noise stability deficit.
\newblock Technical Report 1307.2781, arXiv, 2013.

\bibitem{Fei99}
U.~Feige.
\newblock {\em Randomized rounding of semidefinite programs -- variations on
  the {Max-Cut} example}, volume 1761 of {\em Lecture Notes in Computer
  Science}, pages 189--196.
\newblock Springer, 1999.

\bibitem{FGL+96}
U.~Feige, S.~Goldwasser, L.~Lov{\'a}sz, S.~Safra, and M.~Szegedy.
\newblock Interactive proofs and the hardness of approximating cliques.
\newblock {\em Journal of the ACM}, 43(2):268--292, 1996.

\bibitem{FS02}
U.~Feige and G.~Schechtman.
\newblock On the optimality of the random hyperplane rounding technique for
  {Max-Cut}.
\newblock {\em Randoom Structures and Algorithms}, 20(3):403--440, 2002.

\bibitem{FKN02}
E.~Friedgut, G.~Kalai, and A.~Naor.
\newblock Boolean functions whose {F}ourier transform is concentrated on the
  first two levels and neutral social choice.
\newblock {\em Advances in Applied Mathematics}, 29(3):427--437, 2002.

\bibitem{GK68}
M.~Garman and M.~Kamien.
\newblock The paradox of voting: probability calculations.
\newblock {\em Behavioral Science}, 13(4):306--316, 1968.

\bibitem{Geo10}
K.~Georgiou.
\newblock {\em Integrality gaps for strong linear programming and semidefinite
  programming relaxations}.
\newblock PhD thesis, University of Toronto, 2010.

\bibitem{GW94}
M.~Goemans and D.~Williamson.
\newblock A 0.878 approximation algorithm for {MAX-2SAT} and {MAX-CUT}.
\newblock In {\em Proceedings of the 26th Annual ACM Symposium on Theory of
  Computing}, pages 422--431, 1994.

\bibitem{GW95}
M.~Goemans and D.~Williamson.
\newblock Improved approximation algorithms for maximum cut and satisfiability
  problems using semidefinite programming.
\newblock {\em Journal of the ACM}, 42:1115--1145, 1995.

\bibitem{GV01}
D.~Grigoriev and N.~Vorobjov.
\newblock Complexity of {N}ull- and {P}ositivstellensatz proofs.
\newblock {\em Annals of Pure and Applied Logic}, 113(1):153--160, 2001.

\bibitem{Gro75}
L.~Gross.
\newblock Logarithmic {S}obolev inequalities.
\newblock {\em American Journal of Mathematics}, 97(4):1061--1083, 1975.

\bibitem{Gui52}
G.-T. Guilbaud.
\newblock Les th\'{e}ories de l'int\'{e}r\^{e}t g\'{e}n\'{e}ral et le
  probl\`{e}me logique de l'agr\'{e}gation.
\newblock {\em Economie appliqu\'{e}e}, V(4):501--551, 1952.

\bibitem{Har64}
L.~Harper.
\newblock Optimal assignments of numbers to vertices.
\newblock {\em Journal of the Society for Industrial and Applied Mathematics},
  12(1):131--135, 1964.

\bibitem{Has97}
J.~H{\aa}stad.
\newblock Some optimal inapproximability results.
\newblock In {\em Proceedings of the 29th Annual ACM Symposium on Theory of
  Computing}, pages 1--10, 1997.

\bibitem{Has01}
J.~H{\aa}stad.
\newblock Some optimal inapproximability results.
\newblock {\em Journal of the ACM}, 48(4):798--859, 2001.

\bibitem{Hil02}
D.~Hilbert.
\newblock Mathematical problems.
\newblock {\em Bulletin of the American Mathematical Society}, 8(10):437--479,
  1902.

\bibitem{Hin10}
M.~Hino.
\newblock Sets of finite perimeter and the {H}ausdorff-{G}auss measure on the
  {W}iener space.
\newblock {\em Journal of Functional Analysis}, 258(5):1656--1681, 2010.

\bibitem{Jan97}
S.~Janson.
\newblock {\em Gaussian Hilbert Spaces}.
\newblock Cambridge University Press, 1997.

\bibitem{Joh06}
D.~Johnson.
\newblock The {NP}-{Completeness} column: the many limits on approximation.
\newblock {\em ACM Transactions on Algorithms}, 2(3):473--489, 2006.

\bibitem{Kal02}
G.~Kalai.
\newblock A {F}ourier-theoretic perspective on the {C}ondorcet paradox and
  {A}rrow's theorem.
\newblock {\em Advances in Applied Mathematics}, 29(3):412--426, 2002.

\bibitem{Kan12a}
D.~Kane.
\newblock The correct exponent for the {G}otsman--{L}inial conjecture.
\newblock Technical Report 1210.1283, arXiv, 2012.

\bibitem{Kar99}
H.~Karloff.
\newblock How good is the {G}oemans--{W}illiamson {MAX} {CUT} algorithm?
\newblock {\em SIAM Journal of Computing}, 29(1):336--350, 1999.

\bibitem{Kar72}
R.~Karp.
\newblock Reducibility among combinatorial problems.
\newblock In {\em Complexity of Computer Computations}, pages 85--103. Plenum
  Press, 1972.

\bibitem{Kho02}
S.~Khot.
\newblock On the power of unique 2-prover 1-round games.
\newblock In {\em Proceedings of the 34th Annual ACM Symposium on Theory of
  Computing}, pages 767--775, 2002.

\bibitem{KKMO04}
S.~Khot, G.~Kindler, E.~Mossel, and R.~O'Donnell.
\newblock Optimal inapproximability results for {MAX-CUT} and other 2-variable
  {CSPs}?
\newblock In {\em Proceedings of the 45th Annual IEEE Symposium on Foundations
  of Computer Science}, pages 146--154, 2004.

\bibitem{KKMO07}
S.~Khot, G.~Kindler, E.~Mossel, and R.~O'Donnell.
\newblock Optimal inapproximability results for {M}ax-{C}ut and other
  $2$-variable {CSP}s?
\newblock {\em SIAM Journal on Computing}, 37(1):319--357, 2007.

\bibitem{KS09}
S.~Khot and R.~Saket.
\newblock {SDP} integrality gaps with local $\ell_1$-embeddability.
\newblock In {\em Proceedings of the 50th Annual IEEE Symposium on Foundations
  of Computer Science}, pages 565--574, 2009.

\bibitem{KV05}
S.~Khot and N.~Vishnoi.
\newblock The {Unique Games Conjecture}, integrality gap for cut problems and
  embeddability of negative type metrics into $\ell_1$.
\newblock In {\em Proceedings of the 46th Annual IEEE Symposium on Foundations
  of Computer Science}, pages 53--62, 2005.

\bibitem{KO12}
G.~Kindler and R.~O'Donnell.
\newblock Gaussian noise sensitivity and {F}ourier tails.
\newblock In {\em Proceedings of the 26th Annual IEEE Conference on
  Computational Complexity}, pages 137--147, 2012.

\bibitem{Kri64}
J.-L. Krivine.
\newblock Anneaux pr{\'e}ordonn{\'e}s.
\newblock {\em Journal d'Analyse Math{\'e}matique}, 12(1):307--326, 1964.

\bibitem{Las00}
J.~Lasserre.
\newblock Optimisation globale et th{\'e}orie des moments.
\newblock {\em Comptes Rendus de l'Acad\'{e}mie des Sciences},
  331(11):929--934, 2000.

\bibitem{Led94}
M.~Ledoux.
\newblock Semigroup proofs of the isoperimetric inequality in {E}uclidean and
  {G}auss space.
\newblock {\em Bulletin des Sciences Math\'ematiques}, 118(6):485--510, 1994.

\bibitem{Led98}
M.~Ledoux.
\newblock A short proof of the {G}aussian isoperimetric inequality.
\newblock In {\em High dimensional probability ({O}berwolfach, 1996)},
  volume~43 of {\em Progress in Probability}, pages 229--232. Birkh\"auser,
  Basel, 1998.

\bibitem{Led13}
M.~Ledoux.
\newblock Remarks on noise sensitivity, {B}rascamp--{L}ieb and {S}lepian
  inequalities.
\newblock http://perso.math.univ-toulouse.fr/ledoux/files/2013/11/noise.pdf,
  2013.

\bibitem{Lev22}
P.~L{\'e}vy.
\newblock {\em Le\c{c}ons d'Analyse Fonctionnelle}.
\newblock Gauthier-Villars, 1922.

\bibitem{MZ10}
R.~Meka and D.~Zuckerman.
\newblock Pseudorandom generators for polynomial threshold functions.
\newblock In {\em Proceedings of the 42nd Annual ACM Symposium on Theory of
  Computing}, pages 427--436, 2010.

\bibitem{MNP12}
M.~{Miranda~Jr.}, M.~Novaga, and D.~Pallara.
\newblock An introduction to {BV} functions in {W}iener spaces.
\newblock Technical Report 1212.5926, arXiv, 2012.

\bibitem{Mos10}
E.~Mossel.
\newblock Gaussian bounds for noise correlation of functions.
\newblock {\em Geometric and Functional Analysis}, 19(6):1713--1756, 2010.

\bibitem{MN12}
E.~Mossel and J.~Neeman.
\newblock Robust optimality of {G}aussian noise stability.
\newblock Technical Report 1210.4126, arXiv, 2012.

\bibitem{MOO05}
E.~Mossel, R.~O'Donnell, and K.~Oleszkiewicz.
\newblock Noise stability of functions with low influences: invariance and
  optimality.
\newblock In {\em Proceedings of the 46th Annual IEEE Symposium on Foundations
  of Computer Science}, pages 21--30, 2005.

\bibitem{MOO10}
E.~Mossel, R.~O'Donnell, and K.~Oleszkiewicz.
\newblock Noise stability of functions with low influences: invariance and
  optimality.
\newblock {\em Annals of Mathematics}, 171(1):295--341, 2010.

\bibitem{MOR+06}
E.~Mossel, R.~O'Donnell, O.~Regev, J.~Steif, and B.~Sudakov.
\newblock Non-interactive correlation distillation, inhomogeneous {Markov}
  chains, and the reverse {Bonami--Beckner} inequality.
\newblock {\em Israel Journal of Mathematics}, 154:299--336, 2006.

\bibitem{OD14}
R.~O'Donnell.
\newblock {\em Analysis of Boolean Functions}.
\newblock Cambridge University Press, 2014.

\bibitem{OW08}
R.~O'Donnell and Y.~Wu.
\newblock An optimal {SDP} algorithm for {M}ax-{C}ut, and equally optimal
  {L}ong {C}ode tests.
\newblock In {\em Proceedings of the 40th Annual ACM Symposium on Theory of
  Computing}, pages 335--344, 2008.

\bibitem{OZ13}
R.~O'Donnell and Y.~Zhou.
\newblock Approximability and proof complexity.
\newblock In {\em Proceedings of the 24th Annual ACM-SIAM Symposium on Discrete
  Algorithms}, pages 1537--1556, 2013.

\bibitem{Par00}
P.~Parrilo.
\newblock {\em Structured Semidefinite Programs and Semialgebraic Geometry
  Methods in Robustness and Optimization}.
\newblock PhD thesis, California Institute of Technology, 2000.

\bibitem{Pen46}
L.~Penrose.
\newblock The elementary statistics of majority voting.
\newblock {\em Journal of the Royal Statistical Society}, 109(1):53--57, 1946.

\bibitem{Rag08}
P.~Raghavendra.
\newblock Optimal algorithms and inapproximability results for every {CSP}?
\newblock In {\em Proceedings of the 40th Annual ACM Symposium on Theory of
  Computing}, pages 245--254, 2008.

\bibitem{RS09b}
P.~Raghavendra and D.~Steurer.
\newblock Integrality gaps for strong {SDP} relaxations of {U}nique {G}ames.
\newblock In {\em Proceedings of the 50th Annual IEEE Symposium on Foundations
  of Computer Science}, pages 575--585, 2009.

\bibitem{Rot75}
V.~Rotar'.
\newblock Limit theorems for multilinear forms and quasipolynomial functions.
\newblock {\em Teoriya Veroyatnostei i ee Primeneniya}, 20(3):527--546, 1975.

\bibitem{Sch48}
E.~Schmidt.
\newblock Die {B}runn-{M}inkowskische {U}ngleichung und ihr {S}piegelbild sowie
  die isoperimetrische {E}igenschaft der {K}ugel in der euklidischen und
  nichteuklidischen {G}eometrie. {I}.
\newblock {\em Mathematische Nachrichten}, 1:81--157, 1948.

\bibitem{She99}
W.~Sheppard.
\newblock On the application of the theory of error to cases of normal
  distribution and normal correlation.
\newblock {\em Philosophical Transactions of the Royal Society of London,
  Series A}, 192:101--167, 531, 1899.

\bibitem{Sho87}
N.~Shor.
\newblock Class of global minimum bounds of polynomial functions.
\newblock {\em Cybernetics}, 23(6):731--734, 1987.

\bibitem{Ste73}
G.~Stengle.
\newblock A {N}ullstellensatz and a {P}ositivstellensatz in semialgebraic
  geometry.
\newblock {\em Mathematische Annalen}, 207(2):87--97, 1973.

\bibitem{ST78}
V.~Sudakov and B.~Tsirel'son.
\newblock Extremal properties of half-spaces for spherically invariant
  measures.
\newblock {\em Journal of Soviet Mathematics}, 9(1):9--18, 1978.
\newblock Originally published in Russian in 1974.

\bibitem{TSSW00}
L.~Trevisan, G.~Sorkin, M.~Sudan, and D.~Williamson.
\newblock Gadgets, approximation, and linear programming.
\newblock {\em SIAM Journal on Computing}, 29(6):2074--2097, 2000.

\end{thebibliography}

\end{document}